\documentclass[11pt]{amsart}
\usepackage{graphicx}
\usepackage{amssymb}
\usepackage{adjustbox}
\usepackage{amsthm}
\usepackage{accents}
\usepackage{amsmath}
\usepackage{mathdots}
\usepackage{dynkin-diagrams}
\usepackage{enumerate}
\usepackage{enumitem}
\usepackage{geometry}
\usepackage{mathtools}
\usepackage[utf8]{inputenc}
\usepackage{leftidx}
\newtheorem{theorem}{Theorem}[section]
\newtheorem{proposition}[theorem]{Proposition}
\newtheorem{lemma}[theorem]{Lemma}
\newtheorem{corollary}[theorem]{Corollary}
\newtheorem{conjecture}[theorem]{Conjecture}
\theoremstyle{definition}
\newtheorem{definition}[theorem]{Definition}
\newtheorem{example}[theorem]{Example}
\newtheorem{remark}[theorem]{Remark}
\usepackage{float}
\usepackage{hyperref}
\usetikzlibrary{shapes.geometric}

\usepackage[
  backend=bibtex,
  style=alphabetic
]{biblatex}

\usepackage{tikz}
\usepackage{tikz-cd}

\definecolor{ForestGreen}{RGB}{34,139,34}

\usepackage{stackengine}
\newtheorem{innerthm}{Theorem}
\newenvironment{maintheorem}[1]
  {\renewcommand\theinnerthm{#1}\innerthm}
  {\endinnerthm}

\usetikzlibrary{matrix}

\title[Triangulated polygons and $\mathbf{Y}$-frieze patterns]{Triangulated polygons and $\mathbf{Y}$-frieze patterns}
\author{Hin Chung Henry Tsang}
\author{Jon Wilson}
\address{Jeremiah Horrocks Institute, University of Lancashire,
  Preston PR1 2HE, UK}
\email{hchtsang@lancashire.ac.uk, jwilson30@lancashire.ac.uk}

\date{}

\begin{document}

\begin{abstract}
In the spirit of Conway and Coxeter \cite{CC}, we classify all $\mathbf{Y}$-frieze patterns of type $A_n$. In particular, we settle a conjecture made by de Saint Germain that all such $\mathbf{Y}$-frieze patterns arise from Conway-Coxeter frieze patterns \cite{germain2023frieze}. Moreover, our approach naturally leads to the enumeration of these $\mathbf{Y}$-frieze patterns in terms of Fuss-Catalan numbers.
\end{abstract}
\maketitle

\section{Introduction}

In 1971, Coxeter introduced frieze patterns in \cite{Cox}. A frieze pattern of width $w$ is an array of $w+6$ staggered rows of entries with $3$ fixed rows above and below, and entries between being positive integers, such that every ``diamond'' \[\begin{matrix}
    &N&\\
    W&&E\\
    &S&
\end{matrix}\]
satisfies the following ``diamond rule'':

\[WE=1+NS\]

Two years later, in \cite{CC}, Conway and Coxeter showed that there is a one-to-one correspondence between frieze patterns of width $w$ and triangulations of a $w+3$-gon.

Interest was re-ignited in these classical frieze patterns in the early 2000's by the emergence of Fomin and Zelevinsky's \textit{cluster algebras} \cite{FZ}. Indeed, building on earlier work with Schiffler \cite{caldero2006quivers}, in 2006 Caldero and Chapoton established deep connections between these cluster algebras and the representation theory of quivers \cite{caldero2006cluster}. In particular, for quivers of type $A$, they showed cluster variables admit a representation-theoretic interpretation in terms of indecomposable objects in the associated cluster category, with the clusters themselves corresponding to triangulations of a polygon.

Furthermore, the Auslander--Reiten quiver of the cluster category provides a representation-theoretic realisation of the associated frieze; the indecomposable objects can be arranged according to the combinatorics of the polygon, and the corresponding cluster characters satisfy the diamond rule defining the frieze pattern. Thus giving a conceptual explanation for the appearance of Conway--Coxeter friezes within the representation theory of type $A$ cluster algebras.

Soon after, these connections were placed in a broader geometric setting by Fomin, Shapiro and Thurston, who equipped each marked (orientable) surface with the structure of a cluster algebra \cite{fomin2008cluster}. In this framework, arcs and triangulations of the surface correspond to cluster variables and clusters respectively, with the type $A$ case again arising from a polygon. In particular, when the polygon is realised as an ideal polygon in the hyperbolic plane, the entries of the associated frieze can also be interpreted geometrically in terms of lambda lengths, with the frieze relation reflecting the famous Ptolemy relation. 

This interplay between friezes, cluster algebras, representation theory and geometry has subsequently led to a variety of generalisations and new classes of friezes, including the notion of \textit{unitary friezes} introduced by Gunawan and Schiffler \cite{gunawan2018frieze}, which built on an earlier definition of Morier-Genoud \cite{MG}. Roughly speaking, these are the friezes that arise from a cluster algebra by specialising the initial cluster variables of some cluster to $1$. \newline

On the other hand, the cluster-algebraic framework also gives rise to a second family of variables, known as $\mathbf{Y}$-variables, which encode the coefficient data of a cluster algebra and evolve under mutation according to the rules of a $\mathbf{Y}$-pattern \cite{fomin2007cluster}. In the same paper, inspired by their earlier work in \cite{fomin2003systems}, Fomin and Zelevinsky highlighted that these $\mathbf{Y}$-patterns generalise the $\mathbf{Y}$-systems of Zamolodchikov \cite{zamolodchikov1991thermodynamic}. This general framework therefore provides a natural setting for considering frieze-like structures governed by such $\mathbf{Y}$-patterns.

In 2023, de Saint Germain pursued this angle; introducing the so called $\mathbf{Y}$-frieze patterns in \cite{germain2023frieze}. A $\mathbf{Y}$-frieze pattern of width $w$ is an array of $w+6$ staggered rows of entries with $3$ fixed rows above and below, and entries between being positive integers, such that every ``diamond'' \[\begin{matrix}
    &N&\\
    W&&E\\
    &S&
\end{matrix}\]
satisfies the following ``$\mathbf{Y}$-diamond rule'':

\[WE=(1+N)(1+S)\]

Furthermore, in the same paper, de Saint Germain showed that every frieze pattern $A$ of width $w$ gives rise to a $\mathbf{Y}$-frieze pattern $B$ of width $w$ by taking the second non-trivial row of $A$ as the first non-trivial row of $B$, the map $A\mapsto B$ is denoted by $p_w$.

The map $p_w$ is in general not injective. In \cite{dSG2}, de Saint Germain defined two frieze patterns $A$ and $A'$ to be $\mathbf{Y}$-equivalent if $p_w(A)=p_w(A')$.

On the other hand, de Saint Germain conjectured in the same paper that $p_w$ is surjective for all $w$. In other words, quoting directly from \cite{dSG2}, there is a bijection \[\tilde{p}_w:\operatorname{Frieze}(w)/\sim_{\mathbf{Y}} \rightarrow \operatorname{YFrieze}(w)\]

The surjectivity of $p_w$ can be interpreted in the following way: Since $p_w$ sends unitary frieze patterns to unitary $\mathbf{Y}$-frieze patterns, the surjectivity of $p_w$ implies that all $\mathbf{Y}$-frieze patterns are unitary. For the details please refer to \cite{germain2023frieze} and \cite{MG}.

In this paper, we start with preliminaries on frieze patterns in Section \ref{definitions}. In Section \ref{Y-equivalent}, we shall provide a combinatorial model of $\operatorname{Frieze}(w)/\sim_{\mathbf{Y}}$. In Section \ref{Y-frieze}, we shall see that $p_w$ is surjective for all $w$.

Our \textbf{main results} are the following:

\begin{maintheorem}{A}[Theorem \ref{main}]\label{intro_main}
Every $\mathbf{Y}$-frieze pattern of width $w$ is the image, under $p_w$, of a frieze pattern of width $w$.
\end{maintheorem}

\begin{maintheorem}{B}[Corollary \ref{main-enumeration}]\label{intro_number}
Let $\lvert\operatorname{YFrieze}(w)\rvert$ denote the number of $\mathbf{Y}$-frieze patterns of width $w$. Then we have:

\[
\lvert\operatorname{YFrieze}(w)\rvert = \begin{cases}
C_{w+1}^{(2)},
&\text{if $w$ is even};\\[1.5em]
C_{w+1}^{(2)}-C_{\frac{w+1}{2}}^{(3)},
&\text{if $w\equiv1\text{ }(\operatorname{mod}4)$};\\[1.5em]
C_{w+1}^{(2)}-C_{\frac{w+1}{2}}^{(3)}-C_{\frac{w+1}{4}}^{(5)},
&\text{if $w\equiv3\text{ }(\operatorname{mod}4)$}.
\end{cases}
\]

\noindent where $C_{p}^{(q)}:=\frac{1}{(q-1)p+1}{qp\choose p}$ are the Fuss-Catalan numbers.
\end{maintheorem}

An independent proof of Theorem \ref{intro_main} has been obtained by Ian Short and Andrei Zabolotskii \cite{short2026all}. Interestingly, our two proofs possess rather different flavours. Indeed, in this paper we follow an inductive argument akin to Conway-Coxeter's original proof regarding the correspondence between friezes and triangulated polygons \cite{CC}, whereas Short and Zabolotskii avoid such induction. Instead, given a $\mathbf{Y}$-frieze $B$, they directly produce a rational frieze $A$ with the property that $p_w(A) = B$, and further argue that, where necessary, $A$ can be rescaled to yield a genuine integral frieze satisfying the same property.

\subsection*{Acknowledgments}
We are grateful to the organisers of the conference `Farey's Legacy in Frieze Patterns and Discrete Geometry', held at the ICMS in Edinburgh during April 2026, for providing such a stimulating environment for mathematical exchange. In particular, we thank Michael Cuntz, Anna Felikson, Antoine de Saint Germain, Ian Short, Pavel Tumarkin, Katie Waddle and Andrei Zabolotskii for many helpful and enjoyable discussions. It was during this conference that we discovered Short and Zabolotskii had independently been working on the same problem concerning de Saint Germain's $p$-map conjecture. We are particularly appreciative of their willingness to share a preliminary version of their proof.

\section{Preliminaries and definitions}\label{definitions}


\begin{definition}[Frieze patterns]
    Let $w \in \mathbb{Z}_{\geq-1}$. A \textbf{\emph{frieze pattern}} $A$ of width $w$ is an integer array $(a_{i,j})_{i,j\in\mathbb{Z},-3\leq j-i\leq w+2}$ such that the following three conditions hold.
    \begin{enumerate}[label=\normalfont(\arabic*)]
        \item For all $i\in\mathbb{Z}$ we have: $$a_{i,i-3}=a_{i,i+w+2}=-1, \quad a_{i,i-1}=a_{i,i+w}=1, \quad \text{and} \quad a_{i,i-2}=a_{i,i+w+1}=0.$$
        \item For all $i,j\in\mathbb{Z}$ such that $0\leq j-i\leq w-1$, $a_{i,j}\in\mathbb{Z}^+$.
        \item For all $i,j\in\mathbb{Z}$ such that $-2\leq j-i\leq w+1$, the \textbf{diamond rule} holds: $$a_{i,j}a_{i+1,j+1}=1+a_{i,j+1}a_{i+1,j}.$$
    \end{enumerate}
    The set of all frieze patterns of width $w$ is denoted by $\operatorname{Frieze}(w)$.
\end{definition}

A frieze pattern $A$ of width $w$ can be visualised as $w+6$ staggered rows of entries: \vspace{5mm}

\adjustbox{scale=0.8}{$\begin{matrix}
    \cdots&&-1&&-1&&-1&&-1&&\cdots\\
    &\cdots&&0&&0&&0&&0&&\cdots\\
    &&\cdots&&1&&1&&1&&1&&\cdots\\
    &&&\cdots&&a_{1,1}&&a_{2,2}&&a_{3,3}&&a_{4,4}&&\cdots\\
    &&&&\cdots&&a_{1,2}&&a_{2,3}&&a_{3,4}&&a_{4,5}&&\cdots\\
    &&&&&\ddots&&\ddots&&\ddots&&\ddots&&\ddots&&\ddots\\
    &&&&&&\cdots&&a_{1,w}&&a_{2,w+1}&&a_{3,w+2}&&a_{4,w+3}&&\cdots\\
    &&&&&&&\cdots&&1&&1&&1&&1&&\cdots\\
    &&&&&&&&\cdots&&0&&0&&0&&0&&\cdots\\
    &&&&&&&&&\cdots&&-1&&-1&&-1&&-1&&\cdots
\end{matrix}$}\vspace{5mm}

The rows other than the top $3$ rows and the bottom $3$ rows are called \textbf{nontrivial} rows. Since the first non-trivial row determines the entire frieze pattern, we call it the \textbf{quiddity sequence} of $A$. Furthermore, as shown in Section 2 of \cite{CC}, each frieze pattern of width $w$ has (translational) period $w+3$. Consequently, instead of considering the entire quiddity sequence, it is often convenient to consider the \textbf{quiddity cycle} $q(A):= (a_{1,1}, a_{2,2}, \ldots, a_{w+3,w+3})$.

In the same paper, Conway and Coxeter made the striking observation that frieze patterns of width $w$ are completely classified by the triangulations of a $(w+3)$-gon. Their result is encapsulated in the following theorem.

\begin{theorem}[Section 3 in \cite{CC}]
\label{Conway-Coxeter-thm}

For $w \in \mathbb{Z}_{\geq 1}$, let $P$ be a $(w+3)$-gon whose vertices have been cyclically ordered from $1,\ldots, w+3$. Furthermore, let $T$ be a triangulation of $P$, and for each vertex $i$ of $P$ we define $a_i(T)$ to be the number of triangles in $T$ adjacent to $i$.

Then the following map is a well-defined bijection:

\[
\begin{array}{c@{\quad}c@{\quad}l}
\phi :\quad
\left\{
\begin{gathered}
\text{Triangulations of \hspace{0.1mm}}\\[-2pt]
\text{a $(w+3)$-gon}
\end{gathered}
\right\}
&
\longrightarrow
&
\operatorname{Frieze}(w)
\\[8pt]
\hspace{10mm} T
&
\longmapsto
&
\begin{gathered}
{\scriptstyle A \in \operatorname{Frieze}(w)\text{ such that}}\\[-2pt]
{\scriptstyle q(A)=\left(a_1(T),\ldots,a_{w+3}(T)\right).}
\end{gathered}
\end{array}
\]

\end{theorem}

\begin{example}
Let us consider the case when $w=3$. Then for the triangulation $T$ of the labelled hexagon shown in Figure \ref{Fig-Conway-Coxeter-thm} we have $\left(a_1(T),\ldots,a_{w+3}(T)\right) = (4,1,2,2,2,1)$. Consequently, Theorem \ref{Conway-Coxeter-thm} guarantees there is a unique frieze pattern $A$ of width $3$ whose quiddity cycle is $(4,1,2,2,2,1)$. We depict this frieze pattern on the right of Figure \ref{Fig-Conway-Coxeter-thm}.
\end{example}

\begin{figure}
\label{Fig-Conway-Coxeter-thm}
\[
\begin{tikzpicture}[baseline=(current bounding box.center)]

\begin{scope}[xshift=-5.5cm]
  \draw (0,0) circle (1cm);

  \draw (30:1) to (270:1);
  \draw (30:1) to (210:1);
  \draw (30:1) to (150:1);

  \fill (30:1) circle (2pt);
  \fill (90:1) circle (2pt);
  \fill (150:1) circle (2pt);
  \fill (210:1) circle (2pt);
  \fill (270:1) circle (2pt);
  \fill (330:1) circle (2pt);

  \node at (90:1.25)  {$6$};
  \node at (30:1.25)  {$1$};
  \node at (330:1.25) {$2$};
  \node at (270:1.25) {$3$};
  \node at (210:1.25) {$4$};
  \node at (150:1.25) {$5$};
\end{scope}

\begin{scope}[xshift=-4.7cm]
  \node at (1.5,0) {\Large$\overset{\phi}{\longmapsto}$};
\end{scope}

\begin{scope}[xshift=3.5cm]
  \matrix (frieze) [
    matrix of math nodes,
    row sep=.3em,
    column sep=.5em,
    row 2/.style={xshift=1.5em},
    row 4/.style={xshift=1.5em},
    row 6/.style={xshift=1.5em}
  ] {
    \cdots \hspace{-5mm} & 0 & 0 & 0 & 0 & 0 & 0 & 0 & 0 & \hspace{-5mm} \cdots \\
    \cdots \hspace{-5mm} & 1 & 1 & 1 & 1 & 1 & 1 & 1 & 1 & \hspace{-5mm} \cdots \\
    \cdots \hspace{-5mm} & 1 & 4 & 1 & 2 & 2 & 2 & 1 & 4 & \hspace{-5mm} \cdots \\
    \cdots \hspace{-5mm} & 3 & 3 & 1 & 3 & 3 & 1 & 3 & 3 & \hspace{-5mm} \cdots \\
    \cdots \hspace{-5mm} & 2 & 2 & 2 & 1 & 4 & 1 & 2 & 2 & \hspace{-5mm} \cdots \\
    \cdots \hspace{-5mm} & 1 & 1 & 1 & 1 & 1 & 1 & 1 & 1 & \hspace{-5mm} \cdots \\
    \cdots \hspace{-5mm} & 0 & 0 & 0 & 0 & 0 & 0 & 0 & 0 & \hspace{-5mm} \cdots \\
  };

  \draw[red, thick, rounded corners]
    (frieze-3-3.north west) rectangle
    (frieze-3-8.south east);
\end{scope}

\end{tikzpicture}
\]
\caption{An illustration of Theorem \ref{Conway-Coxeter-thm} in action. Note that, for the given triangulation $T$ of the labelled hexagon on the left, the vertices $1,2,3,4,5,6$ are adjacent to $4,1,2,2,2,1$ triangles in $T$, respectively.
Therefore, Theorem \ref{Conway-Coxeter-thm} tells us that $\phi(T)$ is the unique frieze pattern $A$ such that $q(A) = (4,1,2,2,2,1)$.}
\end{figure}
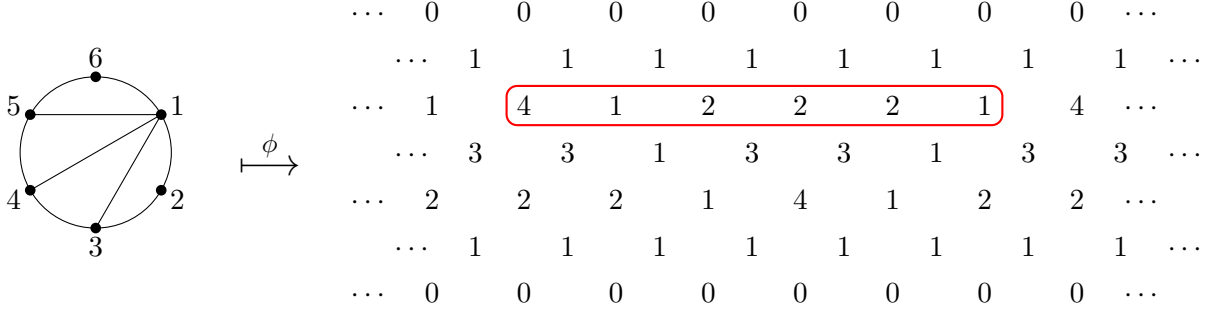

\begin{definition}[$\mathbf{Y}$-frieze patterns]
    Let $w \in \mathbb{Z}_{\geq-1}$. A \emph{$\mathbf{Y}$-frieze pattern} $B$ of width $w$ is an integer array $(b_{i,j})_{i,j\in\mathbb{Z},-3\leq j-i\leq w+2}$ such that the following three conditions hold.
    \begin{enumerate}[label=\normalfont(\arabic*)]
        \item For all $i\in\mathbb{Z}$ we have: $$b_{i,i-3}=b_{i,i-1}=b_{i,i+w}=b_{i,i+w+2}=0,, \quad \text{and}  \quad b_{i,i-2}=b_{i,i+w+1}=-1.$$
        \item For all $i,j\in\mathbb{Z}$ such that $0\leq j-i\leq w-1$, $a_{i,j}\in\mathbb{Z}^+$.
        \item For all $i,j\in\mathbb{Z}$ such that $-2\leq j-i\leq w+1$, the \textbf{Y-diamond rule} holds: 
        $$b_{i,j}b_{i+1,j+1}=(1+b_{i,j+1})(1+b_{i+1,j}).$$
    \end{enumerate}
    The set of all $\mathbf{Y}$-frieze patterns of width $w$ is denoted by $\operatorname{YFrieze}(w)$.
\end{definition}

A $\mathbf{Y}$-frieze pattern $B$ of width $w$ can be visualised as $w+6$ staggered rows of entries: \vspace{5mm}

\adjustbox{scale=0.8}{$\begin{matrix}
    \cdots&&0&&0&&0&&0&&\cdots\\
    &\cdots&&-1&&-1&&-1&&-1&&\cdots\\
    &&\cdots&&0&&0&&0&&0&&\cdots\\
    &&&\cdots&&b_{1,1}&&b_{2,2}&&b_{3,3}&&b_{4,4}&&\cdots\\
    &&&&\cdots&&b_{1,2}&&b_{2,3}&&b_{3,4}&&b_{4,5}&&\cdots\\
    &&&&&\ddots&&\ddots&&\ddots&&\ddots&&\ddots&&\ddots\\
    &&&&&&\cdots&&b_{1,w}&&b_{2,w+1}&&b_{3,w+2}&&b_{4,w+3}&&\cdots\\
    &&&&&&&\cdots&&0&&0&&0&&0&&\cdots\\
    &&&&&&&&\cdots&&-1&&-1&&-1&&-1&&\cdots\\
    &&&&&&&&&\cdots&&0&&0&&0&&0&&\cdots
\end{matrix}$} \vspace{5mm}

The rows other than the top $3$ rows and the bottom $3$ rows are called nontrivial rows. Similar to frieze patterns, since the first non-trivial row determines the entire $\mathbf{Y}$-frieze pattern, we call the first non-trivial row the quiddity sequence of $B$. As with ordinary frieze patterns, in \cite{germain2023frieze} de Saint Germain showed that each $\mathbf{Y}$-frieze pattern of width $w$ has (translational) period $w+3$. Consequently, it similarly makes sense to define the \textbf{quiddity cycle} $q(B):= (b_{1,1}, b_{2,2}, \ldots, b_{w+3,w+3})$.

Furthermore, in the same paper, de Saint Germain constructed a map $p_w : \operatorname{Frieze}(w) \rightarrow \operatorname{YFrieze}(w)$ which sends $A\in\operatorname{Frieze}(w)$ to $B\in\operatorname{YFrieze}(w)$ such that $b_{i,i}=a_{i,i+1}$ for all integers $i$. In other words, the quiddity sequence of $B$ (which determines $B$) is the second non-trivial row of $A$.

The map $p_w$ is in general not injective, as we shall see in an example after the following definition:

\begin{definition}[$\mathbf{Y}$-equivalent frieze patterns]
Let $A$ and $\overline{A}$ be distinct frieze patterns such that $p_w(A)=p_w(\overline{A})$, then we say that they are $\mathbf{Y}$-equivalent, denoted by $A \sim_{\mathbf{Y}} \overline{A}$.
    
\end{definition} 

\begin{example}
The frieze patterns $A,\overline{A}$ of width $1$ with $q(A)=(1,2,1,2)$ and $q(\overline{A})=(2,1,2,1)$ are $\mathbf{Y}$-equivalent as their images under $p_1$ have the same quiddity cycle $(1,1,1,1)$.
    
\end{example}

On the other hand, de Saint Germain conjectured the following:

\begin{conjecture}[Conjecture in \cite{germain2023frieze}]
\label{conjecture}

The map $p_w$ is surjective.

\end{conjecture}

We conclude this section with a simple, yet remarkably useful, identity exhibited by $\mathbf{Y}$-friezes which helps verify Conjecture \ref{conjecture}. In particular, given a $\mathbf{Y}$-frieze $B$ of width $w$, the identity below plays a key role in the proofs of Propositions \ref{2-surjection},\ref{3-surjection} and \ref{4-surjection}, which ensures the well-definedness of the construction of a frieze pattern $A$ with the property that $p_w(A) = B$.

\begin{proposition}
\label{product}
Let $B$ be a $\mathbf{Y}$-frieze pattern of width $w \geq 1$ and suppose $A$ is a frieze pattern such that $p_w(A) = B$.

If we let $q(A) = (a_1,a_2,\ldots, a_{w+3})$ and $q(B) = (b_1,b_2,\ldots, b_{w+3})$ denote the quiddity cycles of $A$ and $B$ respectively then: \[
\displaystyle \prod_{i=1}^{w+3}(b_{i} +1)^{(-1)^{i+1}} = \begin{cases}
1,
&\text{if $w$ is odd};\\[1.5em]
a_1^2,
&\text{if $w$ is even}.
\end{cases}
\]
    
\end{proposition}

\begin{proof}

Note that $p_w(A) = B$ immediately implies $b_i = a_ia_{i+1} -1$ for all $1\leq i \leq w+3$. Therefore, 

\[
\displaystyle \prod_{i=1}^{w+3}(b_{i} +1)^{(-1)^{i+1}} = \begin{cases}
\displaystyle \frac{(b_1+1)(b_3+1)\cdots(b_{w+2}+1)}
{(b_2+1)(b_4+1)\cdots(b_{w+3}+1)} = \frac{(a_1a_2)(a_3a_4)\cdots (a_{w+2}a_{w+3})}{(a_2a_3)(a_4a_5)\cdots (a_{w+3}a_{1})}
=1,
&\text{if $w$ is odd};\\[1.5em]
\displaystyle \displaystyle \frac{(b_1+1)(b_3+1)\cdots(b_{w+3}+1)}
{(b_2+1)(b_4+1)\cdots(b_{w+2}+1)} = \frac{(a_1a_2)(a_3a_4)\cdots (a_{w+3}a_{1})}{(a_2a_3)(a_4a_5)\cdots (a_{w+2}a_{w+3})}
=a_1^2,
&\text{if $w$ is even}.
\end{cases}
\]

\end{proof}

\section{$\mathbf{Y}$-equivalent frieze patterns}\label{Y-equivalent}

The following propositions regarding the relationship between width and $\mathbf{Y}$-equivalence of frieze patterns were stated in \cite{germain2023frieze} and \cite{dSG2}, and we shall provide an explicit proof for each of them:

\begin{proposition}[Remark 1.18 in \cite{MG2}]
\label{distinct-implies-odd}
    Let $A$ and $\overline{A}$ be distinct frieze patterns such that $A\sim_\mathbf{Y} \overline{A}$, then $A$ and $\overline{A}$ must be of odd width.
\end{proposition}

\begin{proof}
    Suppose $A$ and $\overline{A}$ are $\mathbf{Y}$-equivalent frieze patterns of even width $w$, with first row given by $(a_i)$ and $(\overline{a}_i)$, respectively. Then by the diamond rule and the equality of the second row of $A$ and $\overline{A}$, for all $i$ we have \begin{align}
    \label{ratioequality1}    
    a_ia_{i+1}=\overline{a}_i\overline{a}_{i+1}.
    \end{align}

    For a contradiction, suppose $n$ is even. Rearranging equation (\ref{ratioequality1}), for all $i$ we see 
    \begin{align}
    \label{ratioequality2}
    \frac{a_i}{\overline{a}_i}=\frac{\overline{a}_{i+1}}{a_{i+1}}=\dots=\frac{\overline{a}_{i+w+3}}{a_{i+w+3}}=\frac{\overline{a}_i}{a_i}
    \end{align}
    where the final inequality follows from the $w+3$ periodicity of the frieze patterns. Consequently, since $a_i$ and $\overline{a}_i$ are positive, we must have $a_i=\overline{a}_i$ for all $i$. Therefore, distinct frieze patterns that are $\mathbf{Y}$-equivalent must be of odd width.
\end{proof}

\begin{definition}
    Let $(a_i)_{i\in \mathbb{Z}}$ and $(\overline{a}_i)_{i\in \mathbb{Z}}$ be two integer sequences. We say $(\overline{a}_i)$ is related to $(a_i)$ by an \textbf{alternating scalar factor} $q \in \mathbb{Q}$ (denoted $(\overline{a}_i):(a_i)=q$) if for all $k\in \mathbb{Z}$ $$\text{$q \cdot a_{2k+1} = \overline{a}_{2k+1}$ \ \ and \ \  $a_{2k} = q \cdot \overline{a}_{2k}$.}$$
    Similarly, if $A$ and $\overline{A}$ are frieze patterns, we say that $\overline{A}$ is related to $A$ by an alternating scalar factor $q\in\mathbb{Q}$ (denoted $\overline{A}:A=q$) if the first row of $\overline{A}$ is related to the first row of $A$ by an alternating scalar factor $q$.
\end{definition}

\begin{proposition}
\label{distinct-bound}
    Let $A$ be a frieze pattern of odd width, then there is at most one frieze pattern $\overline{A}$ distinct from $A$ such that $A\sim_\mathbf{Y} \overline{A}$.
\end{proposition}

\begin{proof}
    By Lemma 7.5 of \cite{CH}, the first row of $A$, denoted by $(a_i)$, contains the subsequence $(1,2)$, $(2,1)$, or $(1,3,1)$.

    Suppose the first row of $A$ contains the subsequence $(1,2)$. Without loss of generality assume $(a_1,a_2)=(1,2)$. If there is a frieze pattern $\overline{A}$ distinct from $A$ such that $A\sim_\mathbf{Y} \overline{A}$, then the first row of $\overline{A}$, denoted by $(\overline{a}_i)$, is uniquely determined. Indeed, as $a_1a_2=\overline{a}_1\overline{a}_2$ implies $(\overline{a}_1,\overline{a}_2)=(2,1)$, then by equation (\ref{ratioequality1}), $\overline{A}:A=2$, and hence is uniquely determined. \newline \indent
    The cases where $(a_i)$ contains the subsequence $(2,1)$ or $(1,3,1)$ follow analogously. In particular, assuming $(a_1, a_2) = (2,1)$ or $(a_1,a_2,a_3) = (1,3,1)$, then $\overline{A}:A=\frac{1}{2}$ or $3$, respectively.
\end{proof}

    Let $A$ be a frieze pattern. Following Propositions \ref{distinct-implies-odd} and \ref{distinct-bound}, we denote by $\overline{A}$ (if it exists) the unique frieze pattern distinct from $A$ such that $A\sim_\mathbf{Y} \overline{A}$. 
    It is also helpful to note that our proofs of Propositions \ref{distinct-implies-odd} and \ref{distinct-bound} did not use the full power of $\overline{A}$ being a frieze pattern; they merely relied on $A$ being a frieze pattern and condition (\ref{ratioequality1}). In particular, Lemma 7.5 of \cite{CH} also gives rise to the following corollary.

    \begin{corollary}
    \label{Cuntz-Holm}
        Let $A$ be a frieze pattern with quiddity sequence $q(A) = (a_i)_{i\in\mathbb{Z}}$ and suppose $(\overline{a}_i)_{i \in \mathbb{Z}}$ is a positive integer sequence such that: $$\text{$a_ia_{i+1} = \overline{a}_i\overline{a}_{i+1}$ \qquad for all $i \in \mathbb{Z}$}.$$ Then $(\overline{a}_i):(a_i) \in \{\frac{1}{3},\frac{1}{2}, 1, 2, 3\}$ and the following holds:
        \begin{enumerate}[label=\normalfont(\arabic*)]
            \item if $(\overline{a}_i):(a_i) \in \{\frac{1}{2}, 2\}$ then the first row of $A$ contains the subsequence $(1,2)$ or $(2,1)$.
            \item if $(\overline{a}_i):(a_i) \in \{\frac{1}{3}, 3\}$ then the first row of $A$ contains the subsequence $(1,3,1)$.
            \item if $(\overline{a}_i):(a_i) = 1$ then the first row of $A$ contains the subsequence $(1,2)$, $(2,1)$ or $(1,3,1)$.
        \end{enumerate}
    \end{corollary}

\begin{definition}[$n$-cut and $n$-glue]
\label{n-cut}
    Let $n \in \{2,3,4\}$ and $(a_i)_{i\in\mathbb{Z}}$ be an integer sequence with period $m \ge n+3$. Given $j \in \{1,\ldots, m\}$ we define a new integer sequence $(a'_i)_{i\in\mathbb{Z}}$, called an \textbf{\boldmath$n$-cut} of $(a_i)_{i\in\mathbb{Z}}$ at $j$, as follows:
    
\begin{enumerate}[label=\normalfont(\arabic*)]
        \item \textbf{$2$(a)-cut}: if $(a_j,a_{j+1})=(1,2)$, then $(a'_i)_{i\in\mathbb{Z}}$ is the $(m-2)$-periodic integer sequence formed by applying the following replacement to each $m$-period of $(a_i)_{i\in\mathbb{Z}}$:
        \[
        (a_{j-1},a_j,a_{j+1},a_{j+2})
        \mapsto
        (a_{j-1}-2,a_{j+2}-1).
        \]
        
        \item \textbf{$2$(b)-cut}: if $(a_j,a_{j+1})=(2,1)$ then $(a'_i)_{i\in\mathbb{Z}}$ is the $(m-2)$-periodic integer sequence formed  by applying the following replacement to each $m$-period of $(a_i)_{i\in\mathbb{Z}}$:
        \[
        (a_{j-1},a_j,a_{j+1},a_{j+2})
        \mapsto
        (a_{j-1}-1,a_{j+2}-2).
        \]
        
        \item \textbf{$3$-cut}: if $(a_j,a_{j+1},a_{j+2})=(1,3,1)$, then $(a'_i)_{i\in\mathbb{Z}}$ is the $(m-3)$-periodic integer sequence formed by applying the following replacement to each $m$-period of $(a_i)_{i\in\mathbb{Z}}$:
        \[
        (a_{j-1},a_j,a_{j+1},a_{j+2},a_{j+3})
        \mapsto
        (a_{j-1}-2,a_{j+3}-2).
        \]

        \item \textbf{$4$(a)-cut}: if $(a_j,a_{j+1},a_{j+2},a_{j+3})=(1,3,1,3)$ then $(a'_i)_{i\in\mathbb{Z}}$ is the $(m-4)$-periodic integer sequence formed  by applying the following replacement to each $m$-period of $(a_i)_{i\in\mathbb{Z}}$:
        \[
        (a_{j-1},a_j,a_{j+1},a_{j+2},a_{j+3},a_{j+4})
        \mapsto
        (a_{j-1}-3,a_{j+4}-1).
        \]

        \item \textbf{$4$(b)-cut}: if $(a_j,a_{j+1},a_{j+2},a_{j+3})=(3,1,3,1)$ then $(a'_i)_{i\in\mathbb{Z}}$ is the $(m-4)$-periodic integer sequence formed  by applying the following replacement to each $m$-period of $(a_i)_{i\in\mathbb{Z}}$:
        \[
        (a_{j-1},a_j,a_{j+1},a_{j+2},a_{j+3},a_{j+4})
        \mapsto
        (a_{j-1}-1,a_{j+4}-3).
        \]
\end{enumerate}
    We define \textbf{\boldmath$n$-glue} as the reverse operation of an $n$-cut. See Figure \ref{fig:n-cut}.
\end{definition}

\begin{remark}\label{cut polygons}
    In Definition \ref{n-cut}, when the sequence $(a_i)_{i\in\mathbb{Z}}$ arises as the first row of a frieze pattern, all of the cuts described correspond to cutting the associated triangulated polygon $T$. In particular, an $n$-cut corresponds to removing a certain triangulated quadrilateral from $T$ if $n=2$; a triangulated pentagon if $n=3$, and triangulated hexagon if $n=4$. See Figure \ref{fig:n-cut} for a visualisation.
\end{remark}

\begin{figure}[H]
\[
\renewcommand{\arraystretch}{1.5}
\begin{array}{|c|c|}
\hline
\textbf{2-glue}
&
\vcenter{
\hbox{
$
\begin{tikzcd}
      \draw (-1,0) arc[start angle=180,end angle=360,radius=10mm] to (1,0);
      \draw [red] (1,0) arc[start angle=0,end angle=180,radius=10mm] to (-1,0);
      \draw [red] (180:1) to (0:1);
      \draw [red] (180:1) to (60:1);
      \fill (0:1) node[right]{a_{j+2}:= a'_{j}\color{red}+1} circle (2pt);
      \fill (180:1) node[left]{a_{j-1} := a'_{j-1}\color{red}+2} circle (2pt);
      \fill (60:1) node[above]{\color{red}2} circle (2pt);
      \fill (120:1) node[above]{\color{red}1} circle (2pt);
  \end{tikzcd}
  \begin{tikzcd}
      \fill (0,0) node{};
  \end{tikzcd}
  \begin{tikzcd}
      \draw (-1,0) arc[start angle=180,end angle=360,radius=10mm] to (1,0);
      \draw [blue] (1,0) arc[start angle=0,end angle=180,radius=10mm] to (-1,0);
      \draw [blue] (180:1) to (0:1);
      \draw [blue] (120:1) to (0:1);
      \fill (0:1) node[right]{a_{j+2}:= a'_{j}\color{blue}+2} circle (2pt);
      \fill (180:1) node[left]{a_{j-1}:=a'_{j-1}\color{blue}+1} circle (2pt);
      \fill (60:1) node[above]{\color{blue}1} circle (2pt);
      \fill (120:1) node[above]{\color{blue}2} circle (2pt);
  \end{tikzcd}
$
}
\vspace{1.5mm}}
\\
\hline

\textbf{3-glue}
&
\vcenter{
\hbox{
$
\begin{tikzcd}
      \draw (-1,0) arc[start angle=180,end angle=360,radius=10mm] to (1,0);
      \draw [ForestGreen] (1,0) arc[start angle=0,end angle=180,radius=10mm] to (-1,0);
      \draw [ForestGreen] (180:1) to (0:1);
      \draw [ForestGreen] (90:1) to (0:1);
      \draw [ForestGreen] (180:1) to (90:1);
      \fill (0:1) node[right]{a_{j+3}:= a_{j}\color{ForestGreen}+2} circle (2pt);
      \fill (180:1) node[left]{a_{j-1} := a'_{j-1}\color{ForestGreen}+2} circle (2pt);
      \fill (45:1) node[above]{\color{ForestGreen}1} circle (2pt);
      \fill (90:1) node[above]{\color{ForestGreen}3} circle (2pt);
      \fill (135:1) node[above]{\color{ForestGreen}1} circle (2pt);
  \end{tikzcd}
$
}
\vspace{1.5mm}}
\\
\hline

\textbf{4-glue}
&
\vcenter{
\hbox{
$
\begin{tikzcd}
      \draw (-1,0) arc[start angle=180,end angle=360,radius=10mm] to (1,0);
      \draw [red] (1,0) arc[start angle=0,end angle=180,radius=10mm] to (-1,0);
      \draw [red] (180:1) to (0:1);
      \draw [red] (180:1) to (108:1);
      \draw [red] (36:1) to (108:1);
      \draw [red] (180:1) to (36:1);
      \fill (0:1) node[right]{a_{j+4}:= a'_{j}\color{red}+1} circle (2pt);
      \fill (180:1) node[left]{a_{j-1}:= a'_{j-1}\color{red}+3} circle (2pt);
      \fill (36:1) node[above]{\color{red}3} circle (2pt);
      \fill (72:1) node[above]{\color{red}1} circle (2pt);
      \fill (108:1) node[above]{\color{red}3} circle (2pt);
      \fill (144:1) node[above]{\color{red}1} circle (2pt);
  \end{tikzcd}
  \begin{tikzcd}
      \fill (0,0) node{};
  \end{tikzcd}
  \begin{tikzcd}
      \draw (-1,0) arc[start angle=180,end angle=360,radius=10mm] to (1,0);
      \draw [blue] (1,0) arc[start angle=0,end angle=180,radius=10mm] to (-1,0);
      \draw [blue] (180:1) to (0:1);
      \draw [blue] (144:1) to (72:1);
      \draw [blue] (0:1) to (72:1);
      \draw [blue] (144:1) to (0:1);
      \fill (0:1) node[right]{a_{j+4}:= a'_{j}\color{blue}+3} circle (2pt);
      \fill (180:1) node[left]{a_{j-1} := a'_{j-1}\color{blue}+1} circle (2pt);
      \fill (36:1) node[above]{\color{blue}1} circle (2pt);
      \fill (72:1) node[above]{\color{blue}3} circle (2pt);
      \fill (108:1) node[above]{\color{blue}1} circle (2pt);
      \fill (144:1) node[above]{\color{blue}3} circle (2pt);
  \end{tikzcd}
$
}
\vspace{1.5mm}}
\\
\hline
\end{array}
\]
    \caption{We illustrate the triangulations motivating the definition of $n$-glue (and $n$-cut) from Definition \ref{n-cut}. In the case of $2$-glues and $4$-glues, red represents a type (a) glue, and blue represents a type (b) glue. For $3$-glues, there is only one type of glue, and we use green to represent this.}
    \label{fig:n-cut}
\end{figure}

Roughly speaking, the following proposition shows that if one performs $2$-cuts or $4$-cuts on a pair of $\mathbf{Y}$-equivalent frieze patterns, then their alternating scalar factor is preserved.

\begin{proposition}\label{double cut}
Let $(a_i)_{i\in\mathbb{Z}}$ and $(\overline{a}_i)_{i\in\mathbb{Z}}$ be $m$-periodic positive integer sequences.
\begin{enumerate}[label=\normalfont(\arabic*)]
    \item  If $(\overline{a}_i):(a_i) \in \{\frac{1}{2}, 2\}$ and for some $j \in [1,m]$ we have $$\text{$(a_j,a_{j+1}) \in \{(1,2), (2,1)\}$ and $(\overline{a}_j,\overline{a}_{j+1})=(a_{j+1},a_j)$},$$
    then let $(a'_i)_{i\in\mathbb{Z}}$ and $(\overline{a}'_i)_{i\in\mathbb{Z}}$ be the respective $2$-cuts of $(a_i)_{i\in\mathbb{Z}}$ and $(\overline{a}_i)_{i\in\mathbb{Z}}$.
    
    \item  If $(\overline{a}_i):(a_i) \in \{\frac{1}{3}, 3\}$ and for some $j \in [1,m]$ we have $$\text{$(a_j,a_{j+1},a_{j+2},a_{j+3}) \in \{(1,3,1,3), (3,1,3,1)\}$ and $(\overline{a}_j,\overline{a}_{j+1},\overline{a}_{j+2},\overline{a}_{j+3})=(a_{j+3},a_{j+2},a_{j+1},a_j)$},$$
    then let $(a'_i)_{i\in\mathbb{Z}}$ and $(\overline{a}'_i)_{i\in\mathbb{Z}}$ be the respective $4$-cuts of $(a_i)_{i\in\mathbb{Z}}$ and $(\overline{a}_i)_{i\in\mathbb{Z}}$.
    \end{enumerate}

    Then, in either case above, we have $(\overline{a}'_i):(a'_i)=(\overline{a}_i):(a_i)$.
\end{proposition}

\begin{proof}
    We shall only prove the case $(a_j,a_{j+1})=(1,2)$ as the proofs of the other cases are similar.

    It suffices to check that $\frac{\overline{a}_{j-1}-1}{a_{j-1}-2}=\frac{\overline{a}_{j-1}}{a_{j-1}}$ and $\frac{\overline{a}_{j+2}-2}{a_{j+2}-1}=\frac{\overline{a}_{j+2}}{a_{j+2}}$. Since $(\overline{a}_i):(a_i)$ is well-defined, we have $\frac{\overline{a}_{j-1}}{a_{j-1}}=\frac{a_j}{\overline{a}_j}=\frac{1}{2}$, hence $2\overline{a}_{j-1}=a_{j-1}$ and therefore $\frac{\overline{a}_{j-1}-1}{a_{j-1}-2}=\frac{\overline{a}_{j-1}-1}{2\overline{a}_{j-1}-2}=\frac{1}{2}=\frac{\overline{a}_{j-1}}{a_{j-1}}$. Similarly $\frac{\overline{a}_{j+2}-2}{a_{j+2}-1}=\frac{\overline{a}_{j+2}}{a_{j+2}}$ and we are done.
\end{proof}

\begin{remark}
\label{product-preservation}
    Note that, in Proposition \ref{double cut} above, if $a_ia_{i+1} = \overline{a}_i\overline{a}_{i+1}$ for all $i \in \mathbb{Z}$ then we also have $a'_ia'_{i+1} = \overline{a}'_i\overline{a}'_{i+1}$ for all $i \in \mathbb{Z}$. Indeed, this follows directly from the fact that $(\overline{a}'_i):(a'_i)=(\overline{a}_i):(a_i)$.
\end{remark}

\begin{proposition}\label{ratio-classification}
    Let $A$ be a frieze pattern with quiddity sequence $q(A) = (a_i)_{i \in \mathbb{Z}}$ and suppose $(\overline{a}_i)_{i \in \mathbb{Z}}$ is a distinct positive integer sequence such that: $$\text{$a_ia_{i+1} = \overline{a}_i\overline{a}_{i+1}$ \qquad for all $i \in \mathbb{Z}$}.$$ Then $(\overline{a}_i):(a_i) \in \left\{\frac{1}{3}, \frac{1}{2}, 2, 3\right\}$ and the following holds:
    \begin{enumerate}[label=\normalfont(\arabic*)]
        \item $(\overline{a}_i):(a_i)\in \left\{\frac{1}{2}, 2\right\}$ $\iff$ there exists $j \in \mathbb{Z}$ such that $(a_j,a_{j+1}) \in \{(1,2), (2,1)\}$ and $(\overline{a}_j,\overline{a}_{j+1})=(a_{j+1},a_j)$.
        
        \item $(\overline{a}_i):(a_i)\in \left\{\frac{1}{3}, 3\right\}$ $\iff$ there exists $j \in \mathbb{Z}$ such that $(a_j,a_{j+1},a_{j+2},a_{j+3}) \in \{(1,3,1,3), (3,1,3,1)\}$ and $(\overline{a}_j,\overline{a}_{j+1},\overline{a}_{j+2},\overline{a}_{j+3})=(a_{j+3},a_{j+2},a_{j+1},a_j)$.
    \end{enumerate}

    Moreover, there exists a frieze pattern $\overline{A}$ such that $q(\overline{A}) = (\overline{a}_i)$.
    
\end{proposition}

\begin{proof}

Directly from Corollary \ref{Cuntz-Holm} we know $(\overline{a}_i):(a_i) \in \left\{\frac{1}{3}, \frac{1}{2}, 2, 3\right\}$ and that (1) holds. Similarly, regarding (2), we know that $(\overline{a}_i):(a_i)\in \left\{\frac{1}{3}, 3\right\}$ $\iff$  $(a_k, a_{k+1},a_{k+2}) = (1,3,1)$ for some $k \in \mathbb{Z}$
    
    Now, if $(a_k, a_{k+1},a_{k+2}) = (1,3,1)$ for some $k \in \mathbb{Z}$ then $(\overline{a}_k,\overline{a}_{k+1},\overline{a}_{k+2})=(3,1,3)$. Furthermore, since $a_{k-1}a_k=\overline{a}_{k-1}\overline{a}_k$ then $a_{k-1}$ is divisible by $3$. Similarly $a_{k+2}$ is also divisible by $3$. In other words, all subsequences of $(a_i)$ of the form $(x,1,3,1,y)$ must have $x$ and $y$ being multiples of $3$. If $x$ or $y$ is $3$ we are done, so for a contradiction suppose $x,y \geq 6$. Note that, the only other subsequences of $(a_i)$ containing a $1$ will be of the form $(x,1,y)$ where $x,y \geq 3$ by (1). 
    
    Now, simultaneously carry out $3$-cuts at all subsequences of the form $(x,1,3,1,y)$, and $1$-cuts (the move corresponding to removing a single triangle in the polygon associated to $A$) on all other subsequences of the form $(x,1,y)$. Note that, after doing so, we obtain the quiddity cycle of a triangulation that has no $1$'s. Contradiction. Therefore, as desired, $(a_i)$ must contain the subsequence $(1,3,1,3)$.
    

Finally, note that the existence of a frieze pattern $\overline{A}$ such that $q(\overline{A}) = (\overline{a}_i)$ follows by induction on the period of $A$. Indeed, there are two bases cases for which the result is trivially true; here $A$ corresponds to either a quadrilateral or hexagon with respective quiddity cycles $(1,2,1,2)$ and $(1,3,1,3,1,3)$.

Now, assume $A$ has period $m>6$ (the case when $m=5$ is vacuous) and that the result holds for every frieze pattern whose period is less than $m$. By Corollary \ref{Cuntz-Holm} we know $(\overline{a}_i):(a_i) \in \left\{\frac{1}{3}, \frac{1}{2}, 2, 3\right\}$. In turn, Proposition \ref{double cut} tells us that we can perform respective $2$-cuts or $4$-cuts on $(a_i)$ and $(\overline{a}_i)$ to yield two sequences with a smaller period: $(a'_i)$ and $(\overline{a}'_i)$. Furthermore, by the equality $(\overline{a}'_i):(a'_i)=(\overline{a}_i):(a_i)$ and Remark \ref{product-preservation}, then $(\overline{a}'_i)$ is the quiddity sequence of some frieze pattern $\overline{A}'$. Finally, since $(\overline{a}_i)$ arises from a $2$-glue or $4$-glue on $(\overline{a}'_i)$, then the result follows.
    
\end{proof}

\begin{definition}

Let $n \in \mathbb{N}_{\geq 3}$ and consider the $n$-gon whose vertices are labeled $1$ to $n$. \newline
Given $k \in \mathbb{N}_{\geq 3}$ such that $k-2$ divides $n-2$\footnote{There are no partitions of an $n$-gon into $k$-gons if $n-2$ is not divisble by $k-2$.},
we define $\mathcal{P}_{n,k}$ to be the set of all partitions of the $n$-gon into $k$-gons. \newline
Moreover, for each $P \in \mathcal{P}_{n,k}$, we let $T_P$ denote the dual graph of $P$, which is a tree in this context. Namely, $T_P$ is the graph whose vertices are the $k$-gons in the partition $P$, and edges between two vertices exist whenever their corresponding $k$-gons share an arc.
    
\end{definition}


%

\begin{definition}
\label{4-triangulations}
    Let $P\in\mathcal{P}_{n,4}$. We construct triangulations $P_1$ and $P_2$ of $P$ as follows.

    Namely, for each $i \in \{1,2\}$, to define $P_i$ we first consider the quadrilateral in $P$ containing the boundary edge connecting $1$ and $2$. We triangulate this quadrilateral by adding the unique diagonal connected to vertex $i$.


    After that, we uniquely triangulate the remaining quadrilaterals such that all added diagonals (including the diagonal connected to $i$) form a connected graph. Since $T_P$ is a tree, this procedure is indeed well-defined and unique.
  
\end{definition}

\begin{definition}
\label{6-triangulations}
    Let $P\in\mathcal{P}_{n,6}$. We construct triangulations $P_1,P_2$ of $P$ in a similar way to Definition \ref{4-triangulations}. Indeed, for each $6$-gon in $P$, we replace it with one of the two triangulations that possess rotational symmetry of order $2$:

    
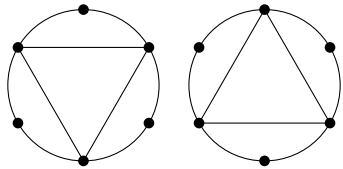
\[\begin{tikzcd}
      \draw (0,0) circle (1cm);
      \draw (30:1) to (150:1);
      \draw (270:1) to (150:1);
      \draw (30:1) to (270:1);
      \fill (30:1) node[above]{} circle (2pt);
      \fill (90:1) node[above]{} circle (2pt);
      \fill (150:1) node[above]{} circle (2pt);
      \fill (210:1) node[below]{} circle (2pt);
      \fill (270:1) node[below]{} circle (2pt);
      \fill (330:1) node[below]{} circle (2pt);
  \end{tikzcd}
  \begin{tikzcd}
      \fill (0,0) node{};
  \end{tikzcd}
  \begin{tikzcd}
      \draw (0,0) circle (1cm);
      \draw (90:1) to (210:1);
      \draw (210:1) to (330:1);
      \draw (330:1) to (90:1);
      \fill (30:1) node[above]{} circle (2pt);
      \fill (90:1) node[above]{} circle (2pt);
      \fill (150:1) node[above]{} circle (2pt);
      \fill (210:1) node[below]{} circle (2pt);
      \fill (270:1) node[below]{} circle (2pt);
      \fill (330:1) node[below]{} circle (2pt);
  \end{tikzcd}\]

  Namely, for each $i \in \{1,2\}$, to define $P_i$, we first consider the $6$-gon in $P$ that contains the boundary edge connecting $1$ and $2$, and choose the unique triangulation of this $6$-gon (with rotational symmetry of order $2$) so that vertex $i$ has $3$ triangles adjacent to it. 

    After that, replace the remaining hexagons with their triangulated counterparts so that all the new diagonals form a connected graph. Since $T_P$ is a tree, this procedure is well-defined and unique.
  
\end{definition}

\begin{proposition}
    Let $i\in\{1,2\}$, $j\in\{1,2,...,n\}$, $P\in\mathcal{P}_{n,4}$. Then if $i\equiv j$ modulo $2$, in each quadrilateral vertex $j$ is adjacent to inside P, it is adjacent to $1$ additional arc on top of those in $P$, otherwise all arcs adjacent to $j$ in $P_i$ are in $P$.
\end{proposition}

\begin{proof}
    We prove by induction on $\frac{n-2}{2}$, which is the number of quadrilaterals. The case $\frac{n-2}{2}=1$ is trivial. Suppose the statement is true for $\frac{n-2}{2}<m-1$ for some $m>1$, when $\frac{n-2}{2}=m$, let $P\in\mathcal{P}_{n,4}$. Since $T_P$ is a tree, we can remove a quadrilateral containing vertices $k,k+1,k+2,k+3$ from $P$ that corresponds to a leaf in $T_P$ to obtain $P'\in\mathcal{P}_{n-2,4}$. To check that the statement holds for $P_1$, observe that since the statement is holds for $P_1'$, it suffices to check that the statement holds for vertices $k,k+1,k+2,k+3$ in $P_1$.
    
    Without loss of generality assume that $k$ is odd. By induction hypothesis $k+3$ is adjacent to only arcs in $P$, so the quadrilateral containing vertices $k,k+1,k+2,k+3$ is triangulated by the arc joining $k$ and $k+2$. Therefore the statement holds for vertices $k,k+1,k+2,k+3$ in $P_1$ and hence all vertices in $P_1$. By similar argument the statement holds for all vertices in $P_2$.
\end{proof}

\begin{proposition}\label{count triangles}
    Let $i\in\{1,2\}$, $j\in\{1,2,...,n\}$, $P\in\mathcal{P}_{n,4}$, and $P(j)$ be the number of quadrilaterals vertex $j$ is adjacent to inside $P$. If $i\equiv j$ modulo $2$, the vertex $j$ is adjacent to $2P(j)$ triangles in $P_i$, otherwise $j$ is adjacent to $P(j)$ triangles in $P_i$.
\end{proposition}

\begin{proof}
    If $i\equiv j$ modulo $2$, then by the previous proposition, $j$ is adjacent to $2$ triangles in each of the $P(j)$ quadrilaterals it is adjacent to inside $P$; similarly if $i\not\equiv j$ modulo $2$, $j$ is adjacent to $1$ triangle in each of the $P(j)$ quadrilaterals it is adjacent to inside $P$.
\end{proof}

Similarly we have the following proposition:

\begin{proposition}
    Let $i\in\{1,2\}$, $j\in\{1,2,...,n\}$, $P\in\mathcal{P}_{n,6}$, and $P(j)$ be the number of hexagons vertex $j$ is adjacent to inside $P$. If $i\equiv j$ modulo $2$, the vertex $j$ is adjacent to $3P(j)$ triangles in $P_i$, otherwise $j$ is adjacent to $P(j)$ triangles in $P_i$.
\end{proposition}

\begin{corollary}
    Let $P\in\mathcal{P}_{n,k}$ where $k=4$ or $6$. The quiddity cycle of $P_2$ is related to that of $P_1$ by an alternating scalar factor $\frac{k}{2}$.
\end{corollary}

\begin{theorem}\label{equivalence}
    Let $A$ be a frieze pattern of width $w$ such that $\overline{A}$ exists. Then \begin{enumerate}[label=\normalfont(\arabic*)]
        \item if $\overline{A}:A=2$, then there exists $P\in\mathcal{P}_{w+3,4}$ such that $P_1$ corresponds to $\overline{A}$ and $P_2$ corresponds to $A$.
        \item if $\overline{A}:A=\frac{1}{2}$, then there exists $P\in\mathcal{P}_{w+3,4}$ such that $P_1$ corresponds to $A$ and $P_2$ corresponds to $\overline{A}$.
        \item if $\overline{A}:A=3$, then there exists $P\in\mathcal{P}_{w+3,6}$ such that $P_1$ corresponds to $\overline{A}$ and $P_2$ corresponds to $A$.
        \item if $\overline{A}:A=\frac{1}{2}$, then there exists $P\in\mathcal{P}_{w+3,6}$ such that $P_1$ corresponds to $A$ and $P_2$ corresponds to $\overline{A}$.
    \end{enumerate}
\end{theorem}

\begin{proof}
    We shall only prove the first case as the other cases are similar. We prove by induction on $w$. The cases $w=-1,0,1$ are trivial and the case $w=2$ is implied by Proposition \ref{distinct-implies-odd}.
    
    Suppose the statement is true for all $w<m$ for some $m\in\mathbb{Z}^+$, when $w=m$, by Proposition \ref{ratio-classification}, the first row of $A$ (denoted by $(a_i)_{i\in\mathbb{Z}}$) contains $(1,2)$ or $(2,1)$ as a subsequence. Without loss of generality assume that $(a_j,a_{j+1})=(1,2)$  (therefore $j$ is odd and the corresponding subsequence in $\overline{A}$ (denoted by $(\overline{a}_i)_{i\in\mathbb{Z}}$) is $(2,1)$). If we perform a 2(a)-cut at $j$ on $(a_i)$ to obtain the sequence $(a'_i)$ and 2(b)-cut at $j$ on $(\overline{a}_i)$ to obtain the sequence $(\overline{a}'_i)$, by Proposition \ref{double cut}, the $(\overline{a}'_i):(a'_i)=2$. By Remark \ref{cut polygons}, $(\overline{a}'_i)$ and $(a'_i)$ are quiddity cycles of $w+1$-gons, therefore they are the first rows of some frieze patterns of width $w-2$, say $A'$ and $\overline{A}'$ respectively. By induction hypothesis, there is $P'\in\mathcal{P}_{w+1,4}$ such that $P'_1$ corresponds to $\overline{A}'$ and $P'_2$ corresponds to $A'$. Consider $P\in\mathcal{P}_{w+3,4}$ given by gluing a quadrilateral on the edge $j-1,j+2$ of $P'$. Since vertex $j-1$ (which is even) in the triangulation that corresponds to $A$ is connected to arcs not in $P'$, $P_1$ corresponds to $A$ and $P_2$ corresponds to $\overline{A}$.
\end{proof}

\begin{theorem}\label{Y-frieze from frieze numnber}
    Let $\lvert p_{w}(\operatorname{Frieze}(w))\rvert$ denote the cardinality of the image of $p_w$. Namely, it is the number of $\mathbf{Y}$-frieze patterns of width $w$ arising, under $p_w$, from frieze patterns. Then we have:

\[
\lvert p_{w}(\operatorname{Frieze}(w))\rvert = \begin{cases}
C_{w+1}^{(2)},
&\text{if $w$ is even};\\[1.5em]
C_{w+1}^{(2)}-C_{\frac{w+1}{2}}^{(3)},
&\text{if $w\equiv1\text{ }(\operatorname{mod}4)$};\\[1.5em]
C_{w+1}^{(2)}-C_{\frac{w+1}{2}}^{(3)}-C_{\frac{w+1}{4}}^{(5)},
&\text{if $w\equiv3\text{ }(\operatorname{mod}4)$}.
\end{cases}
\]

\noindent where $C_{p}^{(q)}:=\frac{1}{(q-1)p+1}{qp\choose p}$ are the Fuss-Catalan numbers.
\end{theorem}

\begin{proof}
    By Theorem \ref{equivalence}, the number of $\mathbf{Y}$-frieze patterns of width $n$ generated from frieze patterns is given by $|\operatorname{Frieze}(w)|-|\mathcal{P}_{w+3,4}|-|\mathcal{P}_{w+3,6}|$, and the value of each term is calculated in \cite{hilton1991catalan}.
\end{proof}

\section{Classification of $\mathbf{Y}$-frieze patterns}\label{Y-frieze}

We introduce another way to denote a part of the entries in a $\mathbf{Y}$-frieze pattern for the sake of better presentation. Given a $\mathbf{Y}$-frieze pattern, we denote the entries in the first non-trivial row by $b_1,b_2,...$ (so $b_{i}=b_{i,i}$). Moreover, we denote the entries in the diagonal starting from $b_1$ by $f_1,f_2,...$ (so $f_{i}=b_{1,i}$) and from $b_2$ by $g_1,g_2,...$ (so $g_{i}=b_{2,i}$). By convention, we set $f_{-1}=-1,f_0=f_{-2}=0$.

\[
\setlength{\arraycolsep}{2pt}
\begin{matrix}
    0&&0&&0&&0&&0&&0&&0&&0&&0\\
    &-1&&-1&&-1&&-1&&-1&&\cdots&&-1&&-1\\
    &&0&&0&&0&&0&&\cdots&&0&&0\\
    &&&b_1&&b_2&&\dots&&\dots&&b_{n-1}&&b_n\\
    &&&&f_2&&g_2&&\dots&&\\
    &&&&&\ddots&&\ddots&&\iddots\\
    &&&&&&f_{n-2}&&g_{n-2}&&\iddots\\
    &&&&&&&f_{n-1}&&g_{n-1}&&\\
    &&&&&&&&f_n
\end{matrix}
\]

\begin{definition}[$\mathbf{Y}$-$n$-cut]
\label{Y-n-cut}
    Let $n \in \{2,3,4\}$ and $(b_i)_{i\in\mathbb{Z}}$ be an integer sequence with period $m \ge n+3$. Given $j \in \{1,\ldots, m\}$ we define a new integer sequence $(b'_i)_{i\in\mathbb{Z}}$, called an $\mathbf{Y}$-\textbf{\boldmath$n$-cut} of $(b_i)_{i\in\mathbb{Z}}$ at $j$, as follows:
    
\begin{enumerate}[label=\normalfont(\arabic*)]
        \item $\mathbf{Y}$-\textbf{$2$-cut}: if $b_j=1$, then $(b'_i)_{i\in\mathbb{Z}}$ is the $(m-2)$-periodic sequence formed by applying the following replacement to each $m$-period of $(b_i)_{i\in\mathbb{Z}}$:
        \[
        (b_{j-2},b_{j-1},b_j,b_{j+1},b_{j+2})
        \mapsto
        \left(b_{j-2}-\tfrac{2(b_{j-2}+1)}{b_{j-1}+1},\tfrac{(b_{j-1}-1)(b_{j+1}-1)}{2}-1,b_{j+2}-\tfrac{2(b_{j+2}+1)}{b_{j+1}+1}\right).
        \]
        
        \item $\mathbf{Y}$-\textbf{$3$-cut}: if $(b_j,b_{j+1})=(2,2)$, then $(b'_i)_{i\in\mathbb{Z}}$ is the $(m-3)$-periodic integer sequence formed by applying the following replacement to each $m$-period of $(b_i)_{i\in\mathbb{Z}}$:
        \[
        (b_{j-2},b_{j-1},b_j,b_{j+1},b_{j+2},b_{j+3})
        \mapsto
        \left(b_{j-2}-\tfrac{2(b_{j-2}+1)}{b_{j-1}+1},(b_{j-1}-1)(b_{j+2}-1)-1,b_{j+3}-\tfrac{2(b_{j+3}+1)}{b_{j+2}+1}\right).
        \]

        \item $\mathbf{Y}$-\textbf{$4$-cut}: if $(b_j,b_{j+1},b_{j+2})=(2,2,2)$ then $(b'_i)_{i\in\mathbb{Z}}$ is the $(m-4)$-periodic integer sequence formed  by applying the following replacement to each $m$-period of $(b_i)_{i\in\mathbb{Z}}$:
        \[
        (b_{j-2},b_{j-1},b_j,b_{j+1},b_{j+2},b_{j+3},b_{j+4})
        \mapsto
        \left(b_{j-2}-\tfrac{3(b_{j-2}+1)}{b_{j-1}+1},\tfrac{(b_{j-1}-2)(b_{j+3}-2)}{3}-1,b_{j+4}-\tfrac{3(b_{j+4}+1)}{b_{j+3}+1}\right).
        \]

\end{enumerate}

\end{definition}

\begin{remark}
    In the case where $(b_i)_{i\in\mathbb{Z}}$ is the quiddity sequence of a $\mathbf{Y}$-frieze pattern, the sequence resulting from a $\mathbf{Y}$-$n$-cut is the quiddity sequence of a $\mathbf{Y}$-frieze pattern, as we shall see in Propositions \ref{Y-2-cut}, \ref{Y-3-cut} and \ref{Y-4-cut}. Furthermore, for each $n \in \{2,3,4\}$, note that $\mathbf{Y}$-$n$-cuts and $n$-cuts are related by the following commutative diagram:

\[
    \begin{tikzcd}
        \operatorname{Frieze}(w)\arrow[rr,"n\text{-cut}"]\arrow[d,"p_{w}"]&&\operatorname{Frieze}(w')\arrow[d,"p_{w'}"]\\
        \operatorname{YFrieze}(w)\arrow[rr,"\mathbf{Y}\text{-}n\text{-cut}"]&&\operatorname{YFrieze}(w')
    \end{tikzcd}
\]

\noindent where $w' = w - n$.

Indeed, we shall chase the diagram for the $n=2$ case with respect to a $2$(a)-cut, and verify commutativity. In particular, consider a frieze pattern $A$ with quiddity cycle $q(A) = (a_1,a_2,...,a_{w+3})$ such that $(a_j,a_{j+1})=(1,2)$. 

If we perform a $2$(a)-cut at $j$ on $q(A)$, we obtain $(a_{j-2},a_{j-1}-2,a_{j+2}-1,a_{j+3})$ in the $j-2^{th}$ to the $j+1^{th}$ entries, which is part of the quiddity cycle of a frieze pattern $A'$ by Theorem \ref{Conway-Coxeter-thm}. Furthermore, note that the $\mathbf{Y}$-frieze pattern $p_{w'}(A')$ arising from $A'$ has: $$(a_{j-2}a_{j-1}-1-2a_{j-2},(a_{j-1}-2)(a_{j+2}-1)-1,a_{j+2}a_{j+3}-1-a_{j+3})$$ in the $j-2^{th}$ to $j^{th}$ entries of its quddity cycle, and all other entries remain unchanged.

On the other hand, consider the $\mathbf{Y}$-frieze pattern $p_w(A)$ arising from $A$. The quiddity cycle of $p_w(A)$ has $(a_{j-2}a_{j-1}-1,a_{j-1}-1,1,2a_{j+2}-1,a_{j+2}a_{j+3}-1)$ in the $j-2^{th}$ to $j+2^{th}$ entries. Performing a $\mathbf{Y}$-$2$-cut at $j$ on $q\left(p_w(A)\right)$, we get: $$\left(a_{j-2}a_{j-1}-1-\frac{2a_{j-2}a_{j-1}}{a_{j-1}},\frac{(a_{j-1}-2)(2a_{j+2}-2)}{2}-1,a_{j+2}a_{j+3}-1-\frac{2a_{j+2}a_{j+3}}{2a_{j+2}}\right)$$ in the $j-2^{th}$ to $j^{th}$ entries, and all other entries remain unchanged. 

Therefore, regardless of the chosen composition in the above diagram, we obtain the same quiddity sequence, and hence the same $\mathbf{Y}$-frieze pattern.

Similarly, one can verify the commutativity of the diagram for all remaining types of cuts.

\end{remark}

We will now build the tools necessary to show that in the first non-trivial row a $\mathbf{Y}$-frieze pattern, one of the $\mathbf{Y}$-$n$-cuts is possible.

\begin{proposition}
\label{y-diagonals}
    Consider a $\mathbf{Y}$-frieze pattern of width $w$. Then for all $n \in \mathbb{Z}$ satisfying $0 \leq n \leq w+3$ we have
    
    \begin{equation}
    \label{y-diagonals1}
        f_n(f_{n-1}+f_{n-2}+1)=(b_nf_{n-1}-(f_{n-2}+1))(f_{n-1}+1).
    \end{equation}

    Equivalently, if $1<n<w+3$ then
\begin{equation}
\label{y-diagonals2}
f_n=\frac{(b_nf_{n-1}-(f_{n-2}+1))(f_{n-1}+1)}{(f_{n-1}+f_{n-2}+1)}.\end{equation}
\end{proposition}

\begin{proof}
    Recall that, by definition, we have $f_{-1}=-1$,$f_0 = 0$ and $f_1 = b_1$. Hence, for $n=0$ or $1$, both sides of equation (\ref{y-diagonals1}) are $0$. Similarly, the equation holds when $n=w+3$. Moreover, for $n=2$, we have $$f_2=b_1b_2-1 = \frac{(b_2f_{1}-(f_{0}+1))(f_{1}+1)}{(f_{1}+f_{0}+1)}$$ where the first equality follows from the $\mathbf{Y}$-diamond rule.

    Now, for $w+3>k \geq 3$, assume equation (\ref{y-diagonals2}) holds for all entries in the ${k-1}^{th}$ row of the $\mathbf{Y}$-frieze. We will show the equation also holds for $n=k$.

    By this induction hypothesis applied to $g_{k-1}$ we have:

    \begin{equation*}
g_{k-1}=\frac{(b_{k}g_{k-2}-(g_{k-3}+1))(g_{k-2}+1)}{(g_{k-2}+g_{k-3}+1)}.
\end{equation*}

Moreover, $f_k =\frac{f_{k-1}g_{k-1}}{g_{k-2}+1}-1$ and $g_{k-3}+1 = \frac{f_{k-2}g_{k-2}}{f_{k-1}+1}$ by the $\mathbf{Y}$-diamond rule, so we obtain:

\begin{align*}
 f_k &=\frac{f_{k-1}g_{k-1}}{g_{k-2}+1}-1 =\frac{f_{k-1}\frac{(b_kg_{k-2}-(g_{k-3}+1))(g_{k-2}+1)}{g_{k-2}+g_{k-3}+1}}{g_{k-2}+1}-1 = \frac{f_{k-1}(b_kg_{k-2}-(g_{k-3}+1))}{g_{k-2}+(g_{k-3}+1)}-1 \\
 &= \frac{f_{k-1}\left(b_kg_{k-2}-\left(\frac{f_{k-2}g_{k-2}}{f_{k-1}+1}\right)\right)}{g_{k-2}+\left(\frac{f_{k-2}g_{k-2}}{f_{k-1}+1}\right)}-1 =\frac{f_{k-1}(b_k(f_{k-1}+1)-f_{k-2})}{f_{k-1}+1+f_{k-2}}-1 \\ &
 =\frac{(b_kf_{k-1}-(f_{k-2}+1))(f_{k-1}+1)}{f_{k-1}+f_{k-2}+1}.
\end{align*}

\end{proof}

From this, we can retrieve the following result:

\begin{corollary}[Theorem 5.2 in \cite{germain2023frieze}]
    We denote by $\mathbf{Y}_w$ the unique glide reflection whose square amounts to a horizontal translation by $w+3$. Every $\mathbf{Y}$-frieze pattern of width $w$ is $\mathbf{Y}_w$-invariant. 
\end{corollary}

\begin{proof}
    Set $n=w+1$, then we have $0=\frac{(b_{w+2}g_{w}-(g_{w-1}+1))(g_{w-2}+1)}{(g_{w}+g_{w-1}+1)}$, therefore $b_{w+2}g_w=g_{w-1}+1$. Now observe that by the $\mathbf{Y}$-diamond rule, $f_wg_w=g_{w-1}+1$. Since $g_w\neq0$, $f_w=b_{w+2}$.
\end{proof}

\begin{remark}
    As a result, $\mathbf{Y}$-frieze patterns are $w+3$-periodic.
\end{remark}

By direct computation from equation \ref{y-diagonals2} we can obtain the following equation:

\begin{corollary}\label{reduction}
    $f_{n}+f_{n-1}+1=\frac{(b_n+1)f_{n-1}(f_{n-1}+1)}{f_{n-1}+f_{n-2}+1}$ for $0\leq n\leq w+3$
\end{corollary}

\begin{proposition}
\label{y-diagonals3}
    Consider a $\mathbf{Y}$-frieze pattern of width $w$. Then for all $n \in \mathbb{Z}$ satisfying $1 \leq n \leq w+3$ we have
    
    \begin{equation}
        f_n=b_nf_{n-1}-\frac{(b_{n-1}f_{n-2}-(f_{n-3}+1))(b_{n}+1)}{b_{n-1}+1}-1.
    \end{equation}
\end{proposition}

\begin{proof}
    Recall that, by definition, we have $f_{-1}=-1$, $f_0=f_{-2} = 0$ and $f_1 = b_1$. The $n=1$ case follows directly from this, and the $n=2$ case additionally follows from the $\mathbf{Y}$-diamond rule: $f_1\cdot b_2 = (1+0)(1+f_2)$.

Now, suppose $n>2$. Then we have:

\begin{alignat*}{2}
\quad f_{n}+1 &= \frac{f_{n-1}g_{n-1}}{g_{n-2}+1}
&&
\parbox[t]{7cm}{\raggedright
by the $\mathbf{Y}$-diamond rule}
\\[1ex]
&= \frac{f_{n-1}(b_ng_{n-2}-(g_{n-3}+1))}{g_{n-2}+g_{n-3}+1}
&&
\parbox[t]{5.5cm}{\raggedright \vspace{-7mm}
since $\frac{g_{n-1}}{g_{n-2}+1}=\frac{b_ng_{n-2}-(g_{n-3}+1)}{g_{n-2}+g_{n-3}+1}$ by Proposition \ref{y-diagonals}}
\\[1ex]
&=\frac{f_{n-1}\left(b_ng_{n-2}-\left(\frac{f_{n-2}g_{n-2}}{f_{n-1}+1}\right)\right)}{g_{n-2}+\frac{f_{n-2}g_{n-2}}{f_{n-1}+1}}
&&
\parbox[t]{7cm}{\raggedright
by the $\mathbf{Y}$-diamond rule}
\\[1ex]
&=\frac{f_{n-1}(b_n(f_{n-1}+1)-f_{n-2})}{f_{n-1}+f_{n-2}+1}
&&
\\
&= b_nf_{n-1}-\frac{(b_{n}+1)f_{n-1}f_{n-2}}{f_{n-1}+f_{n-2}+1}
&&
\\
&= b_nf_{n-1}-\frac{(b_{n}+1)\left(\frac{(b_{n-1}f_{n-2}-(f_{n-3}+1))(f_{n-2}+1)}{(f_{n-2}+f_{n-3}+1)}\right)f_{n-2}}{\left(\frac{(b_{n-1}+1)f_{n-2}(f_{n-2}+1)}{f_{n-2}+f_{n-3}+1}\right)}
&&
\hspace{3mm} \parbox[t]{7cm}{\raggedright
by Corollary \ref{reduction}}
\\[1ex]
&=b_nf_{n-1}-\frac{(b_{n-1}f_{n-2}-(f_{n-3}+1))(b_{n}+1)}{b_{n-1}+1}
&&
\end{alignat*}

This completes the proof.
\end{proof}

\begin{proposition}\label{divisibility}
    Consider a $\mathbf{Y}$-frieze pattern of width $w$ with integer entries. Then for all $k \in \mathbb{Z}$ satisfying $2 \leq 2k \leq w+2$ and $i\in\mathbb{Z}$ we have $\frac{(b_{i}+1)(b_{i+2}+1)\cdots (b_{i+2k}+1)}{(b_{i+1}+1)(b_{i+3}+1)\cdots (b_{i+2k-1}+1)}\in\mathbb{Z}$.
\end{proposition}

\begin{proof}
    Without loss of generality we shall assume $i=1$. 

    Note that for any $k$ with $2\leq 2k\leq w+2$, by Proposition \ref{y-diagonals3} we have: \begin{align}
        f_{2k+1}+f_{2k}+1&=(b_{2k+1}+1)f_{2k}-\frac{(b_{2k}f_{2k-1}-(f_{2k-2}+1))(b_{2k+1}+1)}{b_{2k}+1} \nonumber \\
        &=(b_{2k+1}+1)(f_{2k}-f_{2k-1})+\frac{(f_{2k-1}+f_{2k-2}+1)(b_{2k+1}+1)}{b_{2k}+1}. \label{divisibility-induction}
    \end{align}

    For $k=1$, since $f_0 = 0$ and $f_1 = b_1$, we see equality (\ref{divisibility-induction}) yields: $$f_{3}+f_{2}+1 = (b_{3}+1)(f_{2}-f_{1})+\frac{(b_1+1)(b_{3}+1)}{b_{2}+1}.$$ Therefore, $\frac{(b_1+1)(b_{3}+1)}{b_{2}+1} \in \mathbb{Z}$. Now, we proceed by induction and assume the result holds for all smaller $k$. By iterative substitution of equation (\ref{divisibility-induction}) we obtain:

     \begin{align*}
         f_{2k+1}+f_{2k}+1 &=(b_{2k+1}+1)(f_{2k}-f_{2k-1})+\frac{(b_{2k-1}+1)(b_{2k+1}+1)}{b_{2k}+1}(f_{2k-2}-f_{2k-3}) \hspace{1mm} + \ldots \\ & \cdots+ \frac{(b_3+1)(b_5+1)\cdots(b_{2k+1}+1)}{(b_4+1)(b_6+1)\cdots(b_{2k}+1)}(f_{2}-f_{1}) + \frac{(b_1+1)(b_3+1)\cdots(b_{2k+1}+1)}{(b_2+1)(b_4+1)\cdots(b_{2k}+1)}.
     \end{align*}

    Observe that, by our induction hypothesis, $\frac{(b_{2k-1}+1)(b_{2k+1}+1)}{b_{2k}+1},\cdots,\frac{(b_3+1)(b_5+1)\cdots(b_{2k+1}+1)}{(b_4+1)(b_6+1)\cdots(b_{2k}+1)}$ are all integers. Consequently, $\frac{(b_1+1)(b_3+1)\cdots(b_{2k+1}+1)}{(b_2+1)(b_4+1)\cdots(b_{2k}+1)}\in \mathbb{Z}$ and this completes the proof.
\end{proof}

The following proposition is similar to the above, but for entries along diagonals and involves just three entries:

\begin{proposition}\label{diagonal-divisibility}
    Consider a $\mathbf{Y}$-frieze pattern of width $w$ with integer entries. Then for all $n\in\mathbb{Z}$ satisfying $1\leq n\leq w$ we have $\frac{f_{n-1}f_{n+1}}{f_{n}+1}\in\mathbb{Z}$.
\end{proposition}

\begin{proof}
    Note that $\frac{f_{n-1}f_{n+1}}{f_{n}+1}+f_{n+1}=\frac{(f_{n}+f_{n-1}+1)f_{n+1}}{f_{n}+1}=b_{n+1}f_n-(f_{n-1}+1) \in \mathbb{Z}$, where the last equality follows from applying Proposition \ref{y-diagonals} on $f_{n+1}$.
\end{proof}

The following lemma shows that entries along a diagonal of a $\mathbf{Y}$-frieze pattern remains positive if entries on the first row are not "small enough":

\begin{lemma}\label{growth}
    Consider a closed $\mathbf{Y}$-frieze pattern with positive integer entries such that all entries in the first row are $>1$. Suppose $f_{n-1}\geq f_{n-2}+1$.
    \begin{enumerate}[label=\normalfont(\arabic*)]
        \item If $b_{n}\geq 3$, then $f_{n}\geq f_{n-1}+1$.
        \item If $b_{n}\geq5$, $b_{n+1}=2$, then $f_n\geq 2(f_{n-1}+1)$ and $f_{n+1}\geq f_n+1$.
        \item If $b_{n}=2$, $b_{n+1}\geq5$, then $f_n\geq \frac{f_{n-1}+1}{2}$ and $f_{n+1}\geq f_n+1$.
        \item If $b_{n}=2$, $b_{n+1}=8$, $b_{n+2}=2$, then $f_n\geq \frac{f_{n-1}+1}{2}$, $f_{n+1}\geq 2(f_n+1)$, and $f_{n+2}\geq f_{n+1}+1$.
    \end{enumerate}
\end{lemma}

\begin{proof}
    Suppose $b_n\geq3$, then we have\\ \[f_n=\frac{(b_nf_{n-1}-(f_{n-2}+1))(f_{n-1}+1)}{f_{n-1}+f_{n-2}+1}\geq\frac{(b_nf_{n-1}-f_{n-1})(f_{n-1}+1)}{f_{n-1}+1+f_{n-1}-1}=\frac{(b_n-1)(f_{n-1}+1)}{2}\geq f_{n-1}+1\]

    Suppose $b_{n}\geq5$, $b_{n+1}=2$, then we have \begin{align*}f_n&=\frac{(b_nf_{n-1}-(f_{n-2}+1))(f_{n-1}+1)}{f_{n-1}+f_{n-2}+1}\geq\frac{(b_nf_{n-1}-f_{n-1})(f_{n-1}+1)}{f_{n-1}+f_{n-1}}=\frac{(b_n-1)(f_{n-1}+1)}{2}\geq2(f_{n-1}+1)\\
    f_{n+1}&=\frac{(2f_{n}-(f_{n-1}+1))(f_{n}+1)}{f_{n}+f_{n-1}+1}\geq\frac{\left(2f_n-\frac{f_n}{2}\right)(f_n+1)}{f_n+\frac{f_n}{2}}=f_n+1\end{align*}

Suppose $b_{n}=2$, $b_{n+1}\geq5$, then we have \begin{align*}f_n&=\frac{(2f_{n-1}-(f_{n-2}+1))(f_{n-1}+1)}{f_{n-1}+f_{n-2}+1}\geq\frac{(2f_{n-1}-f_{n-1})(f_{n-1}+1)}{f_{n-1}+f_{n-1}}=\frac{(f_{n-1}+1)}{2}\\f_{n+1}&=\frac{(b_{n+1}f_n-(f_{n-1}+1))(f_n+1)}{f_n+f_{n-1}+1}\geq\frac{(b_{n+1}-2)f_n(f_n+1)}{f_n+2f_n}=\frac{(b_{n+1}-2)(f_n+1)}{3}\geq f_n+1\end{align*}

Suppose $b_{n}=2$, $b_{n+1}=8$, $b_{n+2}=2$, then we have \begin{align*}f_n&=\frac{(2f_{n-1}-(f_{n-2}+1))(f_{n-1}+1)}{f_{n-1}+f_{n-2}+1}\geq\frac{(2f_{n-1}-f_{n-1})(f_{n-1}+1)}{f_{n-1}+f_{n-1}}=\frac{(f_{n-1}+1)}{2}\\
f_{n+1}&=\frac{(b_{n+1}f_n-(f_{n-1}+1))(f_n+1)}{f_n+f_{n-1}+1}\geq\frac{(b_{n+1}-2)f_n(f_n+1)}{f_n+2f_n}=\frac{(b_{n+1}-2)(f_n+1)}{3}= 2(f_n+1)\\
f_{n+2}&=\frac{(b_{n+2}f_{n+1}-(f_{n}+1))(f_{n+1}+1)}{f_{n+1}+f_{n}+1}\geq\frac{\left(2-\frac{1}{2}\right)f_{n+1}(f_{n+1}+1)}{f_{n+1}+\frac{f_{n+1}}{2}}= f_{n+1}+1\end{align*}

\end{proof}

Using this lemma, we can prove the following proposition:

\begin{proposition}\label{can cut}
    Consider a closed $\mathbf{Y}$-frieze pattern $B$ of width $w\geq1$ with positive integer entries. Then the first row of $B$ must contain one of the following:

    \begin{enumerate}[label=\normalfont(\arabic*)]
        \item The entry $1$.
        \item A subsequence of the form $(a,2,2,b)$ such that one of $a,b$ is not congruent to $2$ modulo $3$.
        \item The subsequence $(2,2,2)$.
    \end{enumerate}
\end{proposition}

\begin{proof}
    Suppose the first row of $B$ does not contain the entry $1$, then by the first statement of Lemma \ref{growth}, the first row of $B$ must contain at least one $2$, for otherwise $f_n$ would be positive for all $n$.

    We shall show that the first row of $B$ contains the subsequence $(2,2)$ by contradiction. Assume the contrary is true. For all integers $i$ such that $b_i=2$, $\frac{(b_{i-1}+1)(b_{i+1}+1)}{b_i+1}=\frac{(b_{i-1}+1)(b_{i+1}+1)}{3}$ is an integer by Proposition \ref{divisibility}. Therefore either $b_{i-1}+1$ or $b_{i+1}+1$ is a multiple of $3$. By assumption, $b_{i-1}$ and $b_{i+1}$ cannot be $2$, therefore one of them must be $5$.
    
    In the case where two $2$s are separated by one entry (i.e. $b_{i-1}=b_{i+1}=2$ for some $i$), then $\frac{(b_{i-1}+1)(b_{i+1}+1)}{b_i+1}=\frac{9}{b_i+1}$ is an integer by Proposition \ref{divisibility}. Since $b_i\neq0$ is given and $b_i$ cannot be $2$ by assumption, $b_i=8$. In other words, any entry between two $2$s must be $8$.

    Since $b_1,b_2>1$, $f_2=b_1b_2-1\geq2f_1-1\geq f_1+1$. Now we can repeatedly apply the statements 2,3, and 4 of Lemma \ref{growth} by the argument above to show that $f_i$ is positive for $i\in\mathbb{Z}^+$ by grouping isolated $2$s to adjacent entry that is at least $5$, and grouping two $2$s are separated by one entry to the entry in between (which is $8$). Since $B$ is closed, our assumption that the first row of $B$ does not contain the subsequence $(2,2)$ is false.

    Now consider all subsequence of the first row of $B$ of the form $a,2,2,b$. We shall show that in one of these subsequences either $a$ or $b$ is less than $5$ by contradiction. Assume the contrary, that is, $a,b\geq 5$ for all such subsequences. We can apply the same grouping argument from above to show that $f_i$ is positive for $i\in\mathbb{Z}^+$. For instance, if there is an entry $2$ before $a$ but not after $b$, we can group $(2,a,2,2,b)$ into $(2,a,2)$ and $(2,b)$, since $a=8$ by the argument above, we can apply the fourth and third statement in Lemma \ref{growth}. Since $B$ is closed, either $a$ or $b$ is less than $5$.

    Without loss of generality suppose $a<5$, then $a=2,3,$ or $4$. If $a=2$ then the first row of $B$ contains the subsequence $(2,2,2)$. If $a=3$ or $4$ then the first row of $B$ contains a subsequence of the form $(a,2,2,b)$ such that one of $a,b$ is not congruent to $2$ modulo $3$.
\end{proof}

\begin{proposition}\label{Y-2-cut}
    Let $B = (b_{i,j})$ be a $\mathbf{Y}$-frieze pattern of width $w\geq 4$.
    
    If the first row of $B$ contains the entry $1$, then by performing a $\mathbf{Y}$-$2$-cut on the first row, we obtain the quiddity sequence of a $\mathbf{Y}$-frieze pattern $B' = (b'_{i,j})$ of width $w-2$.
\end{proposition}

\begin{proof}
Without loss of generality, we assume $b_{w+1,w+1}=1$.

Due to glide symmetry, to describe $B'$ it suffices to specify the entries $b'_{i,j}$ for $1 \leq i \leq j \leq w-1$, and we do so as follows:
\begin{align}
    b'_{i,j} &:= b_{i,j} \hspace{-20mm}&\text{for}& \ \ 1 \leq i \leq j \leq w-2. \label{Y2cut-inherit}
    \\[1em]  
    b'_{i,w-1} &:= \frac{b_{i,w-1}b_{i,w+1}}{b_{i,w}+1} \hspace{-20mm}& \text{for}& \ \ 1 \leq i \leq w-1. \label{Y2cut-new}
\end{align}

We begin the proof by showing that the quiddity sequence of $B'$ arises from a $\mathbf{Y}$-$2$-cut on the quiddity sequence of $B$. To this end, observe that by equation (\ref{Y2cut-new})  we know $$b'_{w-1, w-1} := \frac{b_{w-1,w-1}b_{w-1,w+1}}{b_{w-1,w}+1}$$ and by equation (\ref{Y2cut-inherit}), equation (\ref{Y2cut-new}) and glide symmetry we have $$b'_{w,w} := b_{1,w-2} \ \  \text{and} \ \ b'_{w+1,w+1} := b'_{2,w-2} = \frac{b_{2,w-1}b_{2,w+1}}{b_{2,w}+1}.$$

Therefore, since all other entries in the quiddity sequence of $B'$ coincide with those of $B$, it remains only to verify the following (final) equality in order to show that $B'$ arises from a $\mathbf{Y}$-$2$-cut of $B$:

\begin{align*}
&\Big(b'_{w-1,w-1}, \ b'_{w,w}, \ b'_{w+1,w+1}\Big) = \left(\frac{b_{w-1,w-1}b_{w-1,w+1}}{b_{w-1,w}+1}, \ b_{1,w-2}, \ \frac{b_{2,w-1}b_{2,w+1}}{b_{2,w}+1}\right) \\
= &\left(b_{w-1,w-1}-\frac{2(b_{w-1,w-1}+1)}{b_{w,w}+1}, \ \frac{(b_{w,w}-1)(b_{w+2,w+2}-1)}{2}-1, \ b_{w+3,w+3}-\frac{2(b_{w+3,w+3}+1)}{b_{w+2,w+2}+1}\right).
\end{align*}

Now, note that by Proposition \ref{y-diagonals} applied to $f_n := b_{w-1,w+1}$ and using $b_{w+1,w+1} = 1$ we have

\begin{align}
    \frac{b_{w-1,w-1}b_{w-1,w+1}}{b_{w-1,w}+1}&=\frac{b_{w-1,w-1}(b_{w+1,w+1}b_{w-1,w}-(b_{w-1,w-1}+1))}{b_{w-1,w}+b_{w-1,w-1}+1}
    =\frac{b_{w-1,w-1}(b_{w-1,w}-(b_{w-1,w-1}+1))}{b_{w-1,w}+b_{w-1,w-1}+1} \nonumber \\[1em] 
    &=b_{w-1,w-1}-\frac{2b_{w-1,w-1}(b_{w-1,w-1}+1)}{b_{w-1,w}+b_{w-1,w-1}+1}
    =b_{w-1,w-1}-\frac{2b_{w-1,w-1}(b_{w-1,w-1}+1)}{b_{w-1,w-1}b_{w,w}+b_{w-1,w-1}} \nonumber \\[1em] 
    &=b_{w-1,w-1}-\frac{2(b_{w-1,w-1}+1)}{b_{w,w}+1} \label{Y-2-cut-holds}
\end{align}

where the penultimate equality follows from the $\mathbf{Y}$-diamond rule $(b_{w-1,w} +1)(0+1) = b_{w-1,w-1}b_{w,w}$.

By symmetry, $\frac{b_{2,w-1}b_{2,w+1}}{b_{2,w}+1}=b_{w+3,w+3}-\frac{2(b_{w+3,w+3}+1)}{b_{w+2,w+2}+1}$.

Finally, observe that $b_{w,w+2} = b_{1,w-2}$ (due to glide symmetry) belongs to the following triangle of entries in $B$:

\[
\begin{matrix}
0 && 0 && 0 && 0\\
& b_{w,w} && 1 && b_{w+2,w+2} &\\
&& b_{w,w+1} && b_{w+1,w+2} &&\\
&&& b_{1,w-2} &&&
\end{matrix}
\]

Therefore, applying the $\mathbf{Y}$-diamond rule to the lower most diamond above and re-arranging we have $$b_{1,w-2} = \frac{(b_{w,w+1})(b_{w+1,w+2})}{2}-1 = \frac{(b_{w,w}-1)(b_{w+2,w+2}-1)}{2}-1$$

where the final equality follows from two further applications of the $\mathbf{Y}$-diamond rule (with $0$ at the tip of each diamond and $b_{w,w+1}$ and $b_{w+1,w+2}$ at the bottom, respectively). This completes the proof that the first row of $B'$ arises from a $\mathbf{Y}$-$2$-cut on the first row of $B$.

We now turn our attention to proving the $\mathbf{Y}$-diamond rule holds everywhere in $B'$. Note that, to prove this, it suffices to show that the rule holds on each diamond that involves the entry $b'_{i,w-1} := \frac{b_{i,w-1}b_{i,w+1}}{b_{i,w}+1}$ for some $i \in \{1,\ldots, w-1\}$, since these are the only entries not in $B$. For each $i \in \{1,\ldots, w-1\}$ there are two such types of diamonds:

\[
\begin{matrix}
    &&b_{i+1,w-2}&&\\
    &b_{i,w-2}&&\frac{b_{i+1,w-1}b_{i+1,w+1}}{b_{i+1,w}+1}\\
    &&\frac{b_{i,w-1}b_{i,w+1}}{b_{i,w}+1}&&
\end{matrix}
\qquad \text{and} \ \qquad
\begin{matrix}
    &&\frac{b_{i+1,w-1}b_{i+1,w+1}}{b_{i+1,w}+1}&&\\
    &\frac{b_{i,w-1}b_{i,w+1}}{b_{i,w}+1}&&b_{1,i-1}\\
    &&b_{1,i-2}
\end{matrix}
\]



In the interest of showing the $\mathbf{Y}$-diamond rule holds on the left diamond above, first observe that the entries in $B'$ satisfy equation (\ref{y-diagonals2}) for $f_n := \frac{b_{i,w-1}b_{i,w+1}}{b_{i,w}+1}$. Indeed, we see:\\\begin{align*}&\frac{\left(\left(\frac{b_{w-1,w-1}b_{w-1,w+1}}{b_{w-1,w}+1}\right)b_{i,w-2}-(b_{i,w-3}+1)\right)(b_{i,w-2}+1)}{b_{i,w-2}+b_{i,w-3}+1}\\
&=\frac{\left(\left(b_{w-1,w-1}-\frac{2(b_{w-1,w-1}+1)}{b_{w,w}+1}\right)b_{i,w-2}-(b_{i,w-3}+1)\right)(b_{i,w-2}+1)}{b_{i,w-2}+b_{i,w-3}+1} \qquad \text{by equation (\ref{Y-2-cut-holds})} \\
&=b_{i,w-1}-\frac{2(b_{w-1,w-1}+1)b_{i,w-2}(b_{i,w-2}+1)}{(b_{w,w}+1)(b_{i,w-2}+b_{i,w-3}+1)} \qquad \text{by Proposition \ref{y-diagonals} applied to $f_n := b_{i,w-1}$} \\ 
&=b_{i,w-1}-\frac{2(b_{i,w-1}+b_{i,w-2}+1)}{b_{w,w}+1} \qquad \text{by Corollary \ref{reduction} applied to $f_n := b_{i,w-1}$}\\
&=b_{i,w-1}-\frac{2b_{i,w-1}(b_{i,w-1}+1)}{b_{i,w}+b_{i,w-1}+1} \qquad \text{by Corollary \ref{reduction} applied to $f_n := b_{i,w}$}\\
&=\frac{b_{i,w-1}(b_{i,w}-(b_{i,w-1}+1))}{b_{i,w}+b_{i,w-1}+1}=\frac{b_{i,w-1}b_{i,w+1}}{b_{i,w}+1}\end{align*}

\noindent where the last equality follows from applying Proposition \ref{y-diagonals} to $f_n := b_{i,w+1}$ and using the assumption that $b_{w+1,w+1}=1$.

Since the $\mathbf{Y}$-diamond rule determines the lower entry of a diamond uniquely, and equation (\ref{y-diagonals2}) was derived from this rule, then any corresponding value satisfying equation (\ref{y-diagonals2}) must coincide with the value prescribed by the $\mathbf{Y}$-diamond rule. Therefore, as $\frac{b_{i,w-1}b_{i,w+1}}{b_{i,w}+1}$ satisfies equation (\ref{y-diagonals2}) and the $\mathbf{Y}$-diamond rule holds on the rows before $\frac{b_{i,w-1}b_{i,w+1}}{b_{i,w}+1}$, it follows that the $\mathbf{Y}$-diamond rule holds for $\frac{b_{i,w-1}b_{i,w+1}}{b_{i,w}+1}$.

The $\mathbf{Y}$-diamond rule holds for the second type of diamond analogously. 

Clearly $b'_{i,w-1}:=\frac{b_{i,w-1}b_{i,w+1}}{b_{i,w}+1}$ is (strictly) positive for all $2 \leq i \leq w-1$. By Proposition \ref{diagonal-divisibility} applied to $f_n:=b_{i,w}$ we see that it is positive.

Therefore, the proposed array $B'$ is a $\mathbf{Y}$-frieze pattern whose quiddity sequence comes from performing a $\mathbf{Y}$-$2$-cut on the first row of $B$.

\end{proof}

\begin{lemma}
\label{Y-2-lemma}
Let $B$ be a $\mathbf{Y}$-frieze pattern of width $w$, and let $B'$ be the $\mathbf{Y}$-frieze pattern obtained from a $\mathbf{Y}$-2-cut on $B$ with respect to the indexing $b_{w+1} := b_{w+1,w+1} = 1$.

Furthermore, let $A'$ be a frieze pattern of width $w-2$ such that $p_{w-2}(A') = B'$, and let $A$ be obtained from $A'$ by a corresponding 2-glue. Let $q(A) = (a_1,\ldots, a_{w+3})$ and $q(A') = (a'_1,\ldots, a'_{w+1})$ be the quiddity cycles of $A$ and $A'$ respectively.
Then the following holds: 
\begin{enumerate}[label=\normalfont(\roman*)]
\item if $A$ is obtained from $A'$ by a 2(a)-glue then $p_w(A) = B$ is equivalent to $b_{w} = a'_{w}+1$,
\item if $A$ is obtained from $A'$ by a 2(b)-glue then $p_w(A) = B$ is equivalent to $b_{w+2} = a'_{w+1}+1$.
\end{enumerate}
\end{lemma}

\begin{proof}
Directly from Remark \ref{cut polygons} and Proposition \ref{double cut}, if $A$ is obtained from $A'$ by a 2(b)-glue then $q(A)=(a'_1,a'_2,...,a'_{w-1},a'_{w}+1,2,1,a'_{w+1}+2).$ 

Since friezes and $\mathbf{Y}$-friezes are determined by their quiddity sequences, we see that $p_{w}(A)=B$ is equivalent to $a_{i}a_{i+1} = b_{i} + 1$ holding for all $1 \leq i \leq w+3$ where $q(B) = (b_1,\ldots, b_{w+3})$. \newline

This may be restated as:

\begin{align*}
b_i =
\begin{cases}
        a'_ia'_{i+1} - 1, &1\leq i\leq w-2\\[0.5em]
        a'_{w-1}(a'_{w} +1)-1, &i=w-1\\[0.5em]
        2a'_{w} + 1, &i=w\\[0.5em]
        1, &i=w+1\\[0.5em]
        a'_{w+1} + 1, &i=w+2\\[0.5em]
        (a'_{w+1}+2)a'_1 -1, &i=w+3
\end{cases}
\end{align*}

Note that $b_{w+1} = 1$ follows immediately from the conditions of a $\mathbf{Y}$-2-cut. On the other hand, by assumption, we have $p_{w-2}(A')=B'$ which is equivalent to $a'_{i}a'_{i+1} = b'_{i} + 1$ holding for all $1 \leq i \leq w+1$ where $q(B') = (b'_1,\ldots, b'_{w+3})$. By the definition of $\mathbf{Y}$-2-cut, this may be restated as:

    \begin{align*}\label{system2}
    a'_{i}a'_{i+1} = \begin{cases}
        \qquad b_i+1, &1\leq i\leq w-2\\[0.5em] 
        \frac{(b_{w}-1)(b_{w-1}+1)}{b_{w}+1}, &i=w-1\\[0.5em] 
        \frac{(b_{w}-1)(b_{w+2}-1)}{2}, &i=w\\[0.5em] 
        \frac{(b_{w+2}-1)(b_{w+3}+1)}{b_{w+2}+1}, &i=w+1
        \end{cases}\tag{*}
    \end{align*}

Immediately from (\ref{system2}) we see $b_{i} = a'_ia'_{i+1} - 1$ for all $1\leq i\leq w-2$
. Now, let us assume $b_{w+2} = a'_{w+1}+1$. Considering $i=w$ from (\ref{system2}) we see $$ 
a'_{w}a'_{w+1} = \frac{(b_{w}-1)(b_{w+2}-1)}{2} = \frac{(b_{w}-1)a'_{w+1}}{2}.$$ Consequently, $b_{w} = 2a'_{w} + 1$.

Similarly, considering $i = w-1$ and $i = w+1$ from (\ref{system2}) and using $b_{w} = 2a'_{w} + 1$ we see: 
\begin{align*}
a'_{w-1}a'_{w} &= \frac{(b_{w}-1)(b_{w-1}+1)}{b_{w}+1} = \frac{a'_{w}(b_{w-1}+1)}{a'_{w} + 1}, \\
a'_{w+1}a'_{1} &= \frac{(b_{w+2}-1)(b_{w+3}+1)}{b_{w+2}+1} = \frac{a'_{w+1}(b_{w+3}+1)}{a'_{w+1}+2}.
\end{align*}

Hence, as desired, we also have $b_{w-1} = a'_{w-1}(a'_{w} +1)-1$ and $b_{w+3} = (a'_{w+1}+2)a'_1 -1$.

The proof of the case when $A$ is obtained from $A'$ by a 2(a)-glue follows analogously.
\end{proof}

\begin{proposition}\label{2-surjection}
     Let $B$ be a $\mathbf{Y}$-frieze pattern of width $w$, and let $B'$ be the $\mathbf{Y}$-frieze pattern obtained from a $\mathbf{Y}$-2-cut on $B$, as in Proposition \ref{Y-2-cut}. If $B'$ is in the range of $p_{w-2}$, then $B$ is in the range of $p_w$.
\end{proposition}

\begin{proof}
    To prove the proposition, we will show there is a frieze pattern $A'$ of width $w-2$ such that: $$ p_{w-2}(A')=B' \qquad \text{and} \qquad p_{w}(A)=B$$
\noindent where $A$ is a frieze pattern obtained by performing a $2$-glue on $A'$.

Directly from Lemma \ref{Y-2-lemma}, in order to show this, it suffices to find a frieze pattern $A'$ such that $p_{w-2}(A') = B'$ and:
    \begin{enumerate}[label=\normalfont(\roman*)]
\item $b_{w} = a'_{w}+1$, if $A$ is obtained from $A'$ by a 2(a)-glue,
\item $b_{w+2} = a'_{w+1}+1$, if $A$ is obtained from $A'$ by a 2(b)-glue.
    \end{enumerate}

Note that, as Examples \ref{Y-2-bad_choice} and \ref{Y-2-bad_choice2} demonstrate, care is required when choosing $A'$ if $B$ has odd width. Indeed, it is possible that $p_{w-2}(A') = B'$ but $p_{w}(A)\neq B$ for any choice of 2-glue relating $A'$ to $A$. The rest of our proof addresses this delicate point; detailing how to avoid this ``bad'' choice of frieze $A'$, and arrive at the desired one. We shall divide into two cases by the parity of $w+1$.

    \textbf{Case 1}: $w+1$ is even.

    Based on the conditions required for a 2(a) and 2(b)-glue, we will now construct two rational sequences $(d_1,\ldots, d_{w+1})$ and $(c_1,\ldots, c_{w+1})$
 and argue that one of them consists entirely of integers, and is our desired $q(A')$.

    Setting $c_{w+1}:=b_{w+2}-1$, we then inductively generate the remaining $c_i$ for $w \geq i \geq 1$ using (\ref{system2}). Doing so, we obtain: 
\[
c_i :=
\begin{cases}
\dfrac{b_w-1}{b_{w+1}+1},
& \text{if } i=w,\\[2ex]
\dfrac{(b_i+1)(b_{i+2}+1)\cdots(b_{w+1}+1)}
{(b_{i+1}+1)(b_{i+3}+1)\cdots(b_w+1)},
& \text{if } 1\le i\le w-1 \text{ and } i \text{ is even},\\[3ex]
\dfrac{(b_i+1)(b_{i+2}+1)\cdots(b_w+1)}
{(b_{i+1}+1)(b_{i+3}+1)\cdots(b_{w+1}+1)},
& \text{if } 1\le i\le w-1 \text{ and } i \text{ is odd},
\end{cases}
\]

\noindent where $b_{w+1}+1 = 2$. Furthermore, note that $(c_1,\ldots, c_{w+1})$ satisfies (\ref{system2}). Indeed, this follows from construction and the equality $\frac{(b_1+1)(b_{3}+1)\cdots(b_w+1)}
{(b_{2}+1)(b_{4}+1)\cdots(b_{w+1}+1)} = \frac{b_{w+3}+1}{b_{w+2}+1}$. In particular, this equality is implied by Proposition \ref{product} applied to $B'$, and the use of (\ref{system2}). We stress this is a valid application of Proposition \ref{product} from the assumption that $p_{w-2}(A') = B'$ for some frieze $A'$.

    Similarly, setting $d_w:=b_w-1$ and following (\ref{system2}) yields:
\[
d_i :=
\begin{cases}
\dfrac{(b_i+1)(b_{i+2}+1)\cdots(b_w+1)}
{(b_{i+1}+1)(b_{i+3}+1)\cdots(b_{w-1}+1)},
& \text{if } 1\le i\le w-1 \text{ and } i \text{ is odd},\\[3ex]
\dfrac{(b_i+1)(b_{i+2}+1)\cdots(b_{w-1}+1)}
{(b_{i+1}+1)(b_{i+3}+1)\cdots(b_w+1)},
& \text{if } 1\le i\le w-1 \text{ and } i \text{ is even},\\[3ex]
\dfrac{(b_{w+2}-1)(b_{w+3}+1)(b_{2}+1)(b_{4}+1)\cdots(b_{w-1}+1)}
{(b_{w+2}+1)(b_{1}+1)(b_{3}+1)\cdots(b_w+1)} = \dfrac{b_{w+2}-1}{2},
& \text{if } i=w+1,\\[2ex]
\end{cases}
\]

\noindent where the equality $\frac{(b_{w+2}-1)(b_{w+3}+1)(b_{2}+1)(b_{4}+1)\cdots(b_{w-1}+1)}
{(b_{w+2}+1)(b_{1}+1)(b_{3}+1)\cdots(b_w+1)} = \frac{b_{w+2}-1}{2}$ follows from Proposition \ref{product} applied to $B'$, and the use of (\ref{system2}). Consequently, $(d_1,\ldots, d_{w+1})$ satisfies (\ref{system2}).

Note that we have the following, where the integrality follows by Proposition \ref{divisibility}:
\begin{equation}
\label{Y-2-oscillating}
\left\{
\begin{aligned}
&2d_i=c_i\in\mathbb{Z}, &&\text{if $i$ is even},\\
&2c_i=d_i\in\mathbb{Z}, &&\text{if $i$ is odd}.
\end{aligned}
\right.
\end{equation} 



Now, consider the following sequence: \begin{equation}
\label{Y-2-candidate-sequence}
d_w,\; c_{w-1},\; d_{w-1},\; c_{w-2},\; d_{w-2},\;\ldots,\; c_1,\; d_1,\; c_{w+1},\; d_{w+1}.
\end{equation}

By (\ref{Y-2-oscillating}), if a given term $d_{2k+1}$ or $c_{2k}$ is even then $c_{2k+1}$ or $d_{2k}$ is an integer, respectively. Hence, moving down the sequence (\ref{Y-2-candidate-sequence}), all entries are guaranteed to be integers until we meet a term $d_{2k+1}$ or $c_{2k}$ which is odd. On the other hand, for any $0 \leq k \leq \frac{w-1}{2}$, if $d_{2k+1}$ is odd, then every subsequent term $d_{2k+1},d_{2k},\ldots d_1 ,d_{w+1}$ is an integer.
 
Indeed, for each such subsequent term $d_i$, if $i$ is odd then $d_i \in \mathbb{Z}$ by (\ref{Y-2-oscillating}), so it remains to consider when $i$ is even. In that case, by Proposition \ref{divisibility} the product $d_i\cdot d_{2k+1} \in \mathbb{Z}$ for $2k+1 \geq i \geq 2$ and an integer multiple of $\frac{(b_{w+2}-1)}{(b_{w+2}+1)}$ if $i = w+1$. Furthermore, since $2d_i \in \mathbb{Z}$ by (\ref{Y-2-oscillating}), then $d_i \in \mathbb{Z}$ by the assumption that $\gcd(d_{2k+1},2) = 1$.

%
%
%

Likewise, for any $1 \leq k \leq \frac{w-1}{2}$, if $c_{2k}$ is odd, then every subsequent term $c_{2k},c_{2k-1},\ldots,c_1, c_{w+1}$ is an integer. 

Therefore, as desired, either $(c_1,\ldots, c_{w+1})$ or $(d_1,\ldots, d_{w+1})$ is an integer sequence, and we choose $q(A'): = (a'_1,\ldots, a'_n)$ to be one such sequence.

    Finally, by assumption, since there exists $(a''_i)_{i\in\mathbb{Z}}$ that also satisfies equations (\ref{system2}), then our integer choice of $(a'_i)_{i\in\mathbb{Z}}$ is a frieze pattern by Proposition \ref{ratio-classification}. This completes the proof of case 1.

    \textbf{Case 2}: $w+1$ is odd.

    Since $w+1$ is odd and by assumption the system of equations (\ref{system2}) has a solution, there is a unique positive solution $(a'_i)_{i\in\mathbb{Z}}$.

    We can compute ${a'_w}^2$ by considering ${a'_w}^2 = \frac{(a'_1a'_2)\cdots (a'_{w-1}a'_w)(a'_wa'_{w+1})}{(a'_{w+1}a'_1)(a'_2a'_3)\cdots (a'_{w-2}a'_{w-1})}$ and using the equations (\ref{system2}) to get:

    \[{a'}_w^2=\frac{(b_{w+2}+1)(b_1+1)(b_3+1)\cdots(b_{w-1}+1)(b_{w}-1)^2}{2(b_{w+3}+1)(b_2+1)(b_4+1)\cdots(b_{w-2}+1)(b_w+1)}\]

    Consider $q=\frac{(b_{w+2}+1)(b_1+1)(b_3+1)\cdots(b_{w-1}+1)}{(b_{w+3}+1)(b_2+1)(b_4+1)\cdots(b_{w-2}+1)(b_w+1)}$. Since $b_{w+1}+1=2$ then $$2q=\frac{(b_{w+2}+1)(b_1+1)(b_3+1)\cdots(b_{w-1}+1)(b_{w+1}+1)}{(b_{w+3}+1)(b_2+1)(b_4+1)\cdots(b_{w-2}+1)(b_w+1)}$$ is a positive integer by Proposition \ref{divisibility}. On the other hand, $$\frac{2}{q}=\frac{(b_{w+1}+1)(b_{w+3}+1)(b_2+1)(b_4+1)\cdots(b_{w-2}+1)(b_w+1)}{(b_{w+2}+1)(b_1+1)(b_3+1)\cdots(b_{w-1}+1)}$$ is also a positive integer by Proposition \ref{divisibility}. Hence $q \in \{\frac{1}{2}, 1, 2\}$.

    If $q=\frac{1}{2}$, then $a'_w=\frac{b_w-1}{2}$ (so $a'_{w+1}=b_{w+2}-1$), and if $q=2$, then $a'_w=b_w-1$. If $q=1$, then $a'_w=\frac{b_w-1}{\sqrt{2}}$ which is absurd since $a'_w$ is a frieze pattern entry and hence must be an integer.

    In summary, for our frieze pattern $A'$ satisfying $p_{w-2}(A') = B'$ we have either $a'_w=b_w-1$ or $a'_{w+1}=b_{w+2}-1$. Therefore, Lemma \ref{Y-2-lemma} guarantees $p_w(A) = B$, and the proof of case 2 is complete.
    
\end{proof}

\begin{example}\label{Y-2-bad_choice}
    Consider the $\mathbf{Y}$-frieze pattern $B$ of width $w=5$ with $q(B)=(2,2,2,3,7,1,2,8)$, then $q(B')=(2,2,2,2,2,2)$. By choosing $A'$ with $q(A')=(1,3,1,3,1,3)$ (check that $p_{w-2}(A')=B'$), there are two choices of $A$: $q(A)$ is either $(1,3,1,3,3,1,2,4)$ or $(1, 3,1,3,2,2,1,5)$, but $q(p_w(A))=(2, 2,2,8,2,1,7,3)$ or $(2,2,2,5,3,1,4,4)$ respectively, both of them are different from $q(B)$ since the ``good'' choice of $A'$ is the one with $q(A')=(3,1,3,1,3,1)$, which is obtained by performing a $2$-(b) cut on $A$ with $q(A)=(3,1,3,1,4,2,1,3)$
\end{example}

\begin{example}\label{Y-2-bad_choice2}
    From Example \ref{Y-2-bad_choice} above, one may get the impression that ``bad'' choices of $A'$  may be salvaged by performing a $2$-glue at a different place (and shifting or reversing the indexing). However, that does not work in general. Indeed, consider the $\mathbf{Y}$-frieze $B$ of width $w = 7$ with $q(B)=(3,3,3,3,1,9,4,1,5,2)$. Then $q(B')=(3,3,3,1,5,5,1,3)$ and there are two friezes $A'$ such that $q(A') = B'$; the ``good'' choice of $A'$ has $q(A')=(4,1,4,1,2,3,2,1)$ and the ``bad'' choice has $q(A')=(2,2,2,2,1,6,1,2)$. Moreover, for the ``bad'' choice of $A'$, one can easily verify there is no frieze $A$ arising from a $2$-glue on $A'$ such that $p_{w}(A) = B$.
\end{example}

\begin{proposition}\label{Y-3-cut}
    Let $B = (b_{i,j})$ be a $\mathbf{Y}$-frieze pattern of width $w\geq 4$.
    
    If the first row of $B$ contains a subsequence of the form $(x,2,2,y)$ such that one of $x,y$ is not congruent to $2$ modulo $3$, then by performing a $\mathbf{Y}$-$3$-cut on the first row, we obtain the quiddity sequence of a $\mathbf{Y}$-frieze pattern $B' = (b'_{i,j})$ of width $w-3$.
    
\end{proposition}

\begin{proof}
    We prove this by constructing staggered rows of positive integers enclosed by rows of $0$s, showing that the first row comes from the $\mathbf{Y}$-$3$-cut and that the $\mathbf{Y}$-diamond rule holds.

Without loss of generality, assume $b_{w,w}=b_{w+1,w+1}=2$ and $x = b_{w-1,w-1} \not\equiv 2 \mod 3$.

Due to glide symmetry, to describe $B'$ it suffices to specify the entries $b'_{i,j}$ for $1 \leq i \leq j \leq w-2$, and we do so as follows:
\begin{align}
    b'_{i,j} &:= b_{i,j} \hspace{-20mm}&\text{for}& \ \ 1 \leq i \leq j \leq w-3. \label{Y3cut-inherit}
    \\[1em]  
    b'_{i,w-2} &:= \frac{b_{i,w-2}b_{i,w+1}(b_{i,w}+b_{i,w-1}+1)}{3(b_{i,w-1}+1)(b_{i,w}+1)} \hspace{-20mm}& \text{for}& \ \ 1 \leq i \leq w-2. \label{Y3cut-new}
\end{align}

As in the proof of Proposition \ref{Y-2-cut}, we begin by showing that the quiddity sequence of $B'$ arises from a $\mathbf{Y}$-$3$-cut on the quiddity sequence of $B$. This amounts to showing the following equality holds:


\begin{align*}
& \Big(b'_{w-2,w-2}, \ b'_{w-1,w-1}, \ b'_{w,w}\Big) \\
:=& \left(\frac{b_{w-2,w-2}b_{w-2,w+1}(b_{w-2,w}+b_{w-2,w-1}+1)}{3(b_{w-2,w-1}+1)(b_{w-2,w}+1)}, \ b_{1,w-3}, \ \frac{b_{2,w-2}b_{2,w+1}(b_{2,w}+b_{2,w-1}+1)}{3(b_{2,w-1}+1)(b_{2,w}+1)}\right) \\
= &\left(b_{w-1}-\frac{2(b_{w-1}+1)}{b_{w}+1},(b_{w}-1)(b_{w+3}-1)-1,b_{w+4}-\frac{2(b_{w+4}+1)}{b_{w+3}+1}\right).
\end{align*}

Firstly, observe that we have: 

\begin{alignat}{2}
&\frac{b_{w-2,w-2}b_{w-2,w+1}(b_{w-2,w}+b_{w-2,w-1}+1)}{3(b_{w-2,w-1}+1)(b_{w-2,w}+1)} \nonumber 
&&   
\\[1em]
&=\frac{b_{w-2,w-2}\frac{(b_{w+1,w+1}b_{w-2,w}-(b_{w-2,w-1}+1))(b_{w-2,w}+1)}{b_{w-2,w}+b_{w-2,w-1}+1}(b_{w-2,w}+b_{w-2,w-1}+1)}{3(b_{w-2,w-1}+1)(b_{w-2,w}+1)} 
&&
\parbox[t]{4.5cm}{\raggedleft
by Proposition \ref{y-diagonals} applied to $f_n := b_{w-2,w+1}$} \nonumber 
\\[1em] 
&=\frac{b_{w-2,w-2}(2b_{w-2,w}-(b_{w-2,w-1}+1))}{3(b_{w-2,w-1}+1)}
&&
\parbox[t]{4.5cm}{\raggedright
since $b_{w+1,w+1}=2$} \nonumber 
\\[1em] 
&= \frac{b_{w-2,w-2}\left(\frac{2(b_{w,w}b_{w-2,w-1}-(b_{w-2,w-2}+1))(b_{w-2,w-1}+1)}{b_{w-2,w-1}+b_{w-2,w-2}+1}-(b_{w-2,w-1}+1)\right)}{3(b_{w-2,w-1}+1)}  
&&
\parbox[t]{4.5cm}{\raggedright by Proposition \ref{y-diagonals} applied to $f_n := b_{w-2,w}$} \nonumber 
\\[1em]
&= \frac{b_{w-2,w-2}\left(2(2b_{w-2,w-1}-(b_{w-2,w-2}+1))-(b_{w-2,w-1}+b_{w-2,w-2}+1))\right)}{3(b_{w-2,w-1}+b_{w-2,w-2}+1)} 
&&
\parbox[t]{4.5cm}{\raggedright \hspace{3mm} since $b_{w,w}=2$} \nonumber 
\\[1em]
&= b_{w-2,w-2}-\frac{2b_{w-2,w-2}(b_{w-2,w-2}+1)}{(b_{w-2,w-1}+b_{w-2,w-2}+1)} \nonumber 
&&
\\[1em]
&= b_{w-2,w-2}-\frac{2(b_{w-2,w-2}+1)}{b_{w-1,w-1}+1}
\label{Y-3-cut-holds}
\end{alignat}

\noindent where the final equality follows from $b_{w-2,w-1} = b_{w-2,w-2}b_{w-1,w-1} -1$ by the $\mathbf{Y}$-diamond rule.

Analogously, we also have $\frac{b_{2,w-2}b_{2,w+1}(b_{2,w}+b_{2,w-1}+1)}{3(b_{2,w-1}+1)(b_{2,w}+1)}=b_{w+4}-\frac{2(b_{w+4}+1)}{b_{w+3}+1}$. 

Finally, observe that $b_{w-1,w+2} = b_{1,w-3}$ (due to glide symmetry) belongs to the following triangle of entries in $B$:

\[
\begin{matrix}
0 && 0 && 0 && 0 && 0\\
& b_{w-1,w-1} && 2 && 2 && b_{w+2,w+2} &\\
&& b_{w-1,w} && b_{w,w+1} && b_{w+1,w+2} &&\\
&&& b_{w-1,w+1} && b_{w,w+2} &&&\\
&&&& b_{1,w-3} &&&&
\end{matrix}
\]

Using the $\mathbf{Y}$-diamond rule to complete the entries of this triangle from the top down, we obtain: $$b_{1,w-3} = (b_{w-1,w-1}-1)(b_{w+2,w+2} -1)$$


Therefore, as desired, the first row of $B'$ arises from a $\mathbf{Y}$-$3$-cut on the first row of $B$.

To prove the $\mathbf{Y}$-diamond rule holds everywhere in $B'$ it suffices to show that the rule holds on each diamond that involves the entry $b'_{i,w-2} := \frac{b_{i,w-2}b_{i,w+1}(b_{i,w}+b_{i,w-1}+1)}{3(b_{i,w-1}+1)(b_{i,w}+1)}$ for some $i \in \{1,\ldots, w-2 \}$, since these are the only entries not in $B$. For each $i \in \{1,\ldots, w-2\}$ there are two such types of diamonds:

\[
\begin{matrix}
    &&b_{i+1,w-3}&&\\
    &b_{i,w-3}&& b'_{i+1,w-2}\\
    &&b'_{i,w-2}&&
\end{matrix}
\qquad \text{and} \ \qquad
\begin{matrix}
    &&b'_{i+1,w-2}&&\\
    &b'_{i,w-2}&&b_{1,i-1}\\
    &&b_{1,i-2}
\end{matrix}
\]

As in Proposition \ref{Y-2-cut}, to show the $\mathbf{Y}$-diamond rule holds on the left diamond, one may equivalently show that the entries in $B'$ satisfy equation (\ref{y-diagonals2}) for $f_n := \frac{b_{i,w-2}b_{i,w+1}(b_{i,w}+b_{i,w-1}+1)}{3(b_{i,w-1}+1)(b_{i,w}+1)}$. We now verify these entries do indeed satisfy equation (\ref{y-diagonals2}):
\begin{alignat*}{2}
&
\frac{\left(\left(\frac{b_{w-2,w-2}b_{w-2,w+1}(b_{w-2,w}+b_{w-2,w-1}+1)}
{3(b_{w-2,w-1}+1)(b_{w-2,w}+1)}\right)b_{i,w-3}-(b_{i,w-4}+1)\right)
(b_{i,w-3}+1)}
{b_{i,w-3}+b_{i,w-4}+1}
&&
\\[1ex]
&=
\frac{\left(\left(b_{w-2,w-2}-\frac{2(b_{w-2,w-2}+1)}
{b_{w-1,w-1}+1}\right)b_{i,w-3}-(b_{i,w-4}+1)\right)
(b_{i,w-3}+1)}
{b_{i,w-3}+b_{i,w-4}+1}
&&
\text{by equation (\ref{Y-3-cut-holds})}
\\[1ex]
&=
b_{i,w-2}
-\frac{2(b_{w-2,w-2}+1)b_{i,w-3}(b_{i,w-3}+1)}
{(b_{w-1,w-1}+1)(b_{i,w-3}+b_{i,w-4}+1)}
&&
\parbox[t]{7cm}{\raggedright
by Proposition \ref{y-diagonals}\\
applied to $f_n:=b_{i,w-2}$}
\\[1ex]
&=
b_{i,w-2}
-\frac{2(b_{i,w-2}+b_{i,w-3}+1)}
{b_{w-1,w-1}+1}
&&
\parbox[t]{7cm}{\raggedright
by Corollary \ref{reduction}\\
applied to $f_n:=b_{i,w-2}$}
\\[1ex]
&=
b_{i,w-2}
-\frac{2b_{i,w-2}(b_{i,w-2}+1)}
{b_{i,w-1}+b_{i,w-2}+1}
=
\frac{b_{i,w-2}(b_{i,w-1}-(b_{i,w-2}+1))}
{b_{i,w-1}+b_{i,w-2}+1}
&&
\parbox[t]{7cm}{\raggedright
by Corollary \ref{reduction}\\
applied to $f_n:=b_{i,w-1}$}
\\[1ex]
&=
\frac{b_{i,w-2}(2b_{i,w-1}-(b_{i,w-2}+1))}
{b_{i,w-1}+b_{i,w-2}+1}
-\frac{b_{i,w-2}b_{i,w-1}}
{b_{i,w-1}+b_{i,w-2}+1}
&&
\\[1ex]
&=
\frac{b_{i,w-2}b_{i,w}}
{b_{i,w-1}+1}
-\frac{b_{i,w-2}(b_{i,w}+b_{i,w-1}+1)}
{3(b_{i,w-1}+1)}
&&
\hspace{-2.5cm}\parbox[t]{6.5cm}{\raggedright
by $b_{w,w}=2$ and both Proposition \ref{y-diagonals} and Corollary \ref{reduction} applied to $f_n:=b_{i,w}$ on each term, respectively}
\\[1ex]
&=
\frac{b_{i,w-2}(2b_{i,w}-(b_{i,w-1}+1))}
{3(b_{i,w-1}+1)}
&&
\\[1ex]
&=
\frac{b_{i,w-2}b_{i,w+1}(b_{i,w}+b_{i,w-1}+1)}
{3(b_{i,w-1}+1)(b_{i,w}+1)}
&&
\hspace{-1.5cm}\parbox[t]{6cm}{\vspace{-0.5cm}\raggedright
by $b_{w+1,w+1}=2$ and Proposition \ref{y-diagonals} applied to $f_n:=b_{i,w+1}$.}
\end{alignat*}

The $\mathbf{Y}$-diamond rule holds for the second type of diamond analogously. 

Clearly $b'_{i,w-2}:=\frac{b_{i,w-2}b_{i,w+1}(b_{i,w}+b_{i,w-1}+1)}{3(b_{i,w-1}+1)(b_{i,w}+1)}$ is (strictly) positive for all $2 \leq i \leq w-2$. To see that it is an integer, observe that \begin{align*}
b'_{i,w-2}&=\frac{b_{i,w-2}b_{i,w+1}(b_{i,w}b_{i,w-1}+b_{i,w}+b_{i,w-1}+1)}{3(b_{i,w-1}+1)(b_{i,w}+1)}-\frac{b_{i,w-2}b_{i,w+1}b_{i,w}b_{i,w-1}}{3(b_{i,w-1}+1)(b_{i,w}+1)} \\[1em] &=\frac{b_{i,w-2}b_{i,w+1}}{3}-\frac13\frac{b_{i,w-2}b_{i,w}}{b_{i,w-1}+1}\frac{b_{i,w-1}b_{i,w+1}}{b_{i,w}+1}
\end{align*} is an integer divided by $3$, by Proposition \ref{diagonal-divisibility}. \newline 
\indent On the other hand, by equation (\ref{Y-3-cut-holds}) we see $b'_{i,w-2}=b_{i,w-2}-\frac{2(b_{i,w-2}+b_{i,w-3}+1)}{b_{w-1,w-1}+1}$ is an integer divided by $b_{w-1,w-1}+1$. Hence, by the assumption that $b_{w-1,w-1} \not\equiv 2 \mod 3$ then $b'_{i,w-2}$ must be an integer.

Therefore, the proposed array $B'$ is a $\mathbf{Y}$-frieze pattern whose quiddity sequence comes from performing a $\mathbf{Y}$-$3$-cut on the first row of $B$.

\end{proof}

\begin{lemma}
\label{Y-3-lemma}
Let $B$ be a $\mathbf{Y}$-frieze pattern of width $w$, and let $B'$ be the $\mathbf{Y}$-frieze pattern obtained from a $\mathbf{Y}$-3-cut on $B$ with respect to the indexing $$(b_{w-1},b_{w},b_{w+1},b_{w+2}) := (b_{w-1,w-1},b_{w,w},b_{w+1,w+1},b_{w+2,w+2}) = (x,2,2,y)$$ where $x$ or $y$ is not congruent to $2$ modulo $3$.

Furthermore, let $A'$ be a frieze pattern of width $w-3$ such that $p_{w-3}(A') = B'$, and let $A$ be obtained from $A'$ by a corresponding 3-glue. Let $q(A) = (a_1,\ldots, a_{w+3})$ and $q(A') = (a'_1,\ldots, a'_{w})$ be the quiddity cycles of $A$ and $A'$ respectively.
Then $p_w(A) = B$ is equivalent to $b_{w-1} = a'_{w-1}+1$.
\end{lemma}

\begin{proof}
Directly from Remark \ref{cut polygons} and Proposition \ref{double cut}, if $A$ is obtained from $A'$ by a 3-glue then $q(A)=(a'_1,a'_2,...,a'_{w-2},a'_{w-1}+2,1,3,1,a'_{w}+2).$ 

As in the proof of Lemma \ref{Y-2-lemma}, writing $q(B) = (b_1,\ldots, b_{w+3})$ then $p_{w}(A) = B$ may be restated as:

\begin{align*}
b_i =
\begin{cases}
        a'_ia'_{i+1} - 1, &1\leq i\leq w-3\\[0.5em]
        a'_{w-2}(a'_{w-1} +2)-1, &i=w-2\\[0.5em]
        a'_{w-1} + 1, &i=w-1\\[0.5em]
        2, &w\leq i\leq w+1\\[0.5em]
        a'_{w} + 1, &i=w+2\\[0.5em]
        (a'_{w}+2)a'_1 -1, &i=w+3
\end{cases}
\end{align*}

Note that $b_{w} = b_{w+1} = 2$ follows immediately from the conditions of a $\mathbf{Y}$-3-cut. On the other hand, by assumption we have $p_{w-2}(A')=B'$ and, directly from the definition of $\mathbf{Y}$-3-cut, this is equivalent to the following:

    \begin{align*}\label{system3}
    a'_{i}a'_{i+1} = \begin{cases}
        \qquad b_i+1, &1\leq i\leq w-3\\[0.5em] 
        \frac{(b_{w-1}-1)(b_{w-2}+1)}{b_{w-1}+1}, &i=w-2\\[0.5em] 
        (b_{w-1}-1)(b_{w+2}-1), &i=w-1\\[0.5em] 
        \frac{(b_{w+2}-1)(b_{w+3}+1)}{b_{w+2}+1}, &i=w
        \end{cases}\tag{**}
    \end{align*}

Therefore, if $b_{w-1}= a'_{w-1}+1$ then considering $i = w-2$, $i=w-1$ and $i=w$ from (\ref{system3}) we respectively obtain: $$b_{w-2} = a'_{w-1}(a'_{w+2}+2) -1, \qquad b_{w+2} = a'_{w+1}+1, \qquad  b_{w+3} = (a'_{w+1}+2)a'_{1} -1.$$

We thus have $p_{w}(A) = B$ and this completes the proof.

\end{proof}

\begin{proposition}\label{3-surjection}
    Let $B$ be a $\mathbf{Y}$-frieze pattern of width $w$, and let $B'$ be the $\mathbf{Y}$-frieze pattern obtained from a $\mathbf{Y}$-3-cut on $B$, as in Proposition \ref{Y-3-cut}. If $B'$ is in the range of $p_{w-3}$, then $B$ is in the range of $p_w$.
\end{proposition}

\begin{proof}

To prove the proposition, we will show there is a frieze pattern $A'$ of width $w-3$ such that: $$ p_{w-2}(A')=B' \qquad \text{and} \qquad p_{w}(A)=B$$
\noindent where $A$ is a frieze pattern obtained by performing a $3$-glue on $A'$. Directly from Lemma \ref{Y-3-lemma}, in order to show this, it suffices to find a frieze pattern $A'$ such that $p_{w-3}(A') = B'$ and $b_{w-1} = a'_{w-1}+1$.

We shall divide into two cases by the parity of $w$.

\textbf{Case 1}: $w$ is even.

We first set $a'_{w-1}:=b_{w-1}-1$ and use this and (\ref{system3}) to iteratively define $a'_i$ for all odd $i$. Similarly, we use $a'_{w}:=b_{w+2}-1$ and (\ref{system3}) to iteratively define $a'_i$ for all even $i$. Doing so we obtain:

\begin{equation}
\label{Y-3-case1-frieze}
a'_i :=
\begin{cases}
\dfrac{(b_i+1)(b_{i+2}+1)\cdots(b_{w-1}+1)}
{(b_{i+1}+1)(b_{i+3}+1)\cdots(b_{w-2}+1)},
& \text{if } 1\le i\le w-3 \text{ and } i \text{ is odd},\\[3ex]
\dfrac{(b_{w+2}+1)(b_1+1)(b_{3}+1)\cdots(b_{i-1}+1)}
{(b_{w+3}+1)(b_{2}+1)(b_{4}+1)\cdots(b_{i-2}+1)},
& \text{if } 2\le i\le w-2 \text{ and } i \text{ is even},
\end{cases}
\end{equation}

Note that, by Proposition \ref{divisibility}, $a_i \in \mathbb{Z}$ for all $1\leq i\leq w$. Moreover, $a'_ia'_{i+1}$ satisfies (\ref{system3}) for all $1\leq i\leq w$. Indeed, the case for $i=w-1$ follows immediately from the definition of $a'_{w-1}$ and $a'_{w}$. Furthermore, by the assumption that $p_{w-3}(A') = B'$ for some frieze $A'$, we may apply Proposition \ref{product} to $B'$ and use (\ref{system3}) to get the equality:

\begin{align}
\label{Y-3-equality}
  &\frac{(b_{1}+1)(b_{3}+1)\cdots (b_{w-1}+1)}{(b_{2}+1)(b_{4}+1)\cdots (b_{w-2}+1)} = \frac{b_{w+3}+1}{b_{w+2}+1}.   
\end{align}

Hence, for any $1 \leq i \leq w-3$, using (\ref{Y-3-case1-frieze}) and factoring out $\frac{(b_i+1)(b_{w+2}+1)}{(b_{w+3}+1)}$ we see that (\ref{Y-3-equality}) yields: $$a'_ia'_{i+1} = \frac{(b_i+1)(b_{w+2}+1)}{(b_{w+3}+1)}\left(\frac{(b_{1}+1)(b_{3}+1)\cdots (b_{w-1}+1)}{(b_{2}+1)(b_{4}+1)\cdots (b_{w-2}+1)}\right) = (b_i +1).$$

Similarly, again by (\ref{Y-3-equality}), we have: \begin{align*}
    a'_{w-2}a'_{w-1} &= \frac{(b_{w+2}+1)}{(b_{w+3}+1)}\left(\frac{(b_{1}+1)(b_{3}+1)\cdots (b_{w-3}+1)}{(b_{2}+1)(b_{4}+1)\cdots (b_{w-4}+1)}\right)\cdot (b_{w-1}+1) \\ &= \frac{(b_{w+2}+1)(b_{w-2}+1)(b_{w+3}+1)(b_{w-1}-1)}{(b_{w+3}+1)(b_{w-1}+1)(b_{w+2}+1)} = \frac{(b_{w-2}+1)(b_{w-1}-1)}{b_{w-1}+1}
\end{align*}

and 

\begin{align*}
    a'_{w}a'_{1} = (b_{w+2}-1)\cdot \left(\frac{(b_{1}+1)(b_{3}+1)\cdots (b_{w-1}+1)}{(b_{2}+1)(b_{4}+1)\cdots (b_{w-2}+1)}\right) = \frac{(b_{w+2}-1)(b_{w+3}+1)}{b_{w+2}+1}.
\end{align*}

Therefore, similar to the proof of Proposition \ref{2-surjection}, since there exists $(a''_i)_{i\in\mathbb{Z}}$ that also satisfies equations (\ref{system3}), then our choice of $(a'_i)_{i\in\mathbb{Z}}$ is a frieze pattern by Proposition \ref{ratio-classification}. This completes the proof of case 1.

    \textbf{Case 2}: $w$ is odd.

     By assumption, the system of equations (\ref{system3}) has a positive integer solution $(a'_i)_{i\in\mathbb{Z}}$. Moreover, as in the proof of Proposition \ref{2-surjection}, since $w$ is odd this solution must be unique. 
     
Considering ${a'_{w-1}}^2=\frac{(a'_1a'_2)\cdots (a'_{w-2}a'_{w-1})(a'_{w-1}a'_{w})}{(a'_{w}a'_1)(a'_{w}a'_1)(a'_2a'_3)\cdots (a'_{w-3}a'_{w-2})}$ and using the equations (\ref{system3}) we get:

\[{a'_{w-1}}^2=\frac{(b_{w+2}+1)(b_1+1)(b_3+1)\cdots(b_{w-2}+1)(b_{w-1}-1)^2}{(b_{w+3}+1)(b_2+1)(b_4+1)\cdots(b_{w-3}+1)(b_{w-1}+1)}\]

Consider $q=\frac{(b_{w+2}+1)(b_1+1)(b_3+1)\cdots(b_{w-2}+1)}{(b_{w+3}+1)(b_2+1)(b_4+1)\cdots(b_{w-3}+1)(b_{w-1}+1)}$. Since $b_{w}+1=3$ then $$3q=\frac{(b_{w+2}+1)(b_1+1)(b_3+1)\cdots(b_{w-2}+1)(b_w+1)}{(b_{w+3}+1)(b_2+1)(b_4+1)\cdots(b_{w-3}+1)(b_{w-1}+1)}$$ is a positive integer by Proposition \ref{divisibility}. On the other hand, since $b_{w+1}+1=3$ then $$\frac{3}{q}=\frac{(b_{w+1}+1)(b_{w+3}+1)(b_2+1)(b_4+1)\cdots(b_{w-3}+1)(b_{w-1}+1)}{(b_{w+2}+1)(b_1+1)(b_3+1)\cdots(b_{w-2}+1)}$$ is also a positive integer by Proposition \ref{divisibility}. Hence $q\in \{ \frac{1}{3}, 1, 3\}$.

If $q=1$, then $a'_{w-1}=b_{w-1}-1$. If $q=\frac{1}{3}$ or $3$ then $a'_{w-1}=\frac{b_{w-1}-1}{\sqrt{3}}$ or $\sqrt{3}(b_{w-1}-1)$, respectively, which is absurd in either scenario since $a'_{w-1}$ is a frieze pattern entry and hence must be an integer.

In summary, for our frieze pattern $A'$ satisfying $p_{w-2}(A') = B'$ we have $a_{w-1}=b_{w-1}-1$. Therefore, Lemma \ref{Y-3-lemma} guarantees $p_w(A) = B$, and the proof of case 2 is complete.
    
\end{proof}

\begin{proposition}\label{Y-4-cut}
    Let $B = (b_{i,j})$ be a $\mathbf{Y}$-frieze pattern of width $w\geq 4$.
    
    If the first row of $B$ contains the subsequence $(2,2,2)$, then by performing a $\mathbf{Y}$-$4$-cut on the first row, we obtain the quiddity sequence of a $\mathbf{Y}$-frieze pattern $B' = (b'_{i,j})$ of width $w-4$.
\end{proposition}

\begin{proof}

 We prove this by constructing staggered rows of positive integers enclosed by rows of $0$s, showing that the first row comes from the $\mathbf{Y}$-$4$-cut and that the $\mathbf{Y}$-diamond rule holds.

Without loss of generality, assume $b_{w-1,w-1}=b_{w,w}=b_{w+1,w+1}=2$.

Due to glide symmetry, to describe $B'$ it suffices to specify the entries $b'_{i,j}$ for $1 \leq i \leq j \leq w-3$, and we do so as follows:
\begin{align}
    b'_{i,j} &:= b_{i,j} \hspace{-20mm}&\text{for}& \ \ 1 \leq i \leq j \leq w-4. \label{Y4cut-inherit}
    \\[1em]  
    b'_{i,w-3} &:= \frac{b_{i,w-3}b_{i,w-1}b_{i,w+1}}{(b_{i,w-2}+1)(b_{i,w}+1)} \hspace{-20mm}& \text{for}& \ \ 1 \leq i \leq w-3. \label{Y4cut-new}
\end{align}

As in the proofs of Proposition \ref{Y-2-cut} and \ref{Y-3-cut}, we begin by showing that the quiddity sequence of $B'$ arises from a $\mathbf{Y}$-$4$-cut on the quiddity sequence of $B$. This amounts to showing the following equality holds:

\begin{align*}
& \Big(b'_{w-3,w-3}, \ b'_{w-2,w-2}, \ b'_{w-1,w-1}\Big) \\
:=& \left(\frac{b_{w-3,w-3}b_{w-3,w-1}b_{w-3,w+1}}{(b_{w-3,w-2}+1)(b_{w-3,w}+1)}, \ b_{1,w-4}, \ \frac{b_{2,w-3}b_{2,w-1}b_{2,w+1}}{(b_{2,w-2}+1)(b_{2,w}+1)}\right) \\
= &\left(b_{{w-3,w-3}}-\frac{3(b_{w-3,w-3}+1)}{b_{w-2,w-2}+1},\frac{(b_{w-2}-2)(b_{w+2}-2)}{3}-1,b_{w+3,w+3}-\frac{3(b_{w+3,w+3}+1)}{b_{w+2,w+2}+1}\right).
\end{align*}

Firstly, applying Proposition \ref{y-diagonals} to $f_n:= b_{w-3,w-1}$ and $f_n:= b_{w-3,w+1}$, and using $b_{w-1,w-1}=b_{w+1,w+1} = 2$ we have: \begin{align}
\label{Y-4-quiddity-equality}
& b'_{w-3,w-3} := \frac{b_{w-3,w-3}b_{w-3,w-1}b_{w-3,w+1}}{(b_{w-3,w-2}+1)(b_{w-3,w}+1)} \nonumber \\[1ex] &= \frac{b_{w-3,w-3}\left(\frac{(b_{w-1,w-1}b_{w-3,w-2} - (b_{w-3,w-3}+1))(b_{w-3,w-2}+1)}{b_{w-3,w-2}+b_{w-3,w-3}+1} \right)\left(\frac{(b_{w+1,w+1}b_{w-3,w} - (b_{w-3,w-1}+1))(b_{w-3,w}+1)}{b_{w-3,w}+b_{w-3,w-1}+1} \right)}{(b_{w-3,w-2}+1)(b_{w-3,w}+1)}
\nonumber \\[1ex] &= b_{w-3,w-3}\frac{(2b_{w-3,w-2} - (b_{w-3,w-3}+1))}{(b_{w-3,w-2}+b_{w-3,w-3}+1)}\frac{(2b_{w-3,w} - (b_{w-3,w-1}+1))}{(b_{w-3,w}+b_{w-3,w-1}+1)}.
\end{align}

Furthermore, note that: \begin{alignat*}{2}
&
\quad \frac{2b_{w-3,w} - (b_{w-3,w-1}+1)}{(b_{w-3,w}+b_{w-3,w-1}+1)}
&&
\\[1ex]
&=
\frac{2\left(\frac{(b_{w,w}b_{w-3,w-1} - (b_{w-3,w-2}+1))(b_{w-3,w-1}+1)}{b_{w-3,w-1}+b_{w-3,w-2}+1}\right) - (b_{w-3,w-1}+1)}{\left(\frac{(b_{w,w}b_{w-3,w-1} - (b_{w-3,w-2}+1))(b_{w-3,w-1}+1)}{b_{w-3,w-1}+b_{w-3,w-2}+1}\right)+(b_{w-3,w-1}+1)}
&&
\qquad \qquad \parbox[t]{7cm}{\raggedright
by Proposition \ref{y-diagonals}\\
applied to $f_n:=b_{w-3,w}$}
\\[1ex]
&=
\frac{2\left(2b_{w-3,w-1} - (b_{w-3,w-2}+1)\right) - (b_{w-3,w-1}+b_{w-3,w-2}+1)}{\left(2b_{w-3,w-1} - (b_{w-3,w-2}+1)\right)+(b_{w-3,w-1}+b_{w-3,w-2}+1)}
&&
\qquad \qquad \parbox[t]{7cm}{\raggedright
by $b_{w,w}=2$}
\\[1ex]
&=
\frac{b_{w-3,w-1}-(b_{w-3,w-2}+1)}{b_{w-3,w-1}}.
\end{alignat*}

Hence, employing Proposition \ref{y-diagonals}
applied to $f_n:=b_{w-3,w-1}$ and using $b_{w-1,w-1} = 2$, we see that equality (\ref{Y-4-quiddity-equality}) may be rewritten as: \begin{alignat}{2}
&
\quad b'_{w-3,w-3} = b_{w-3,w-3}\frac{(2b_{w-3,w-2} - (b_{w-3,w-3}+1))}{(b_{w-3,w-2}+b_{w-3,w-3}+1)}\frac{b_{w-3,w-1}-(b_{w-3,w-2}+1)}{b_{w-3,w-1}}
&&
\nonumber \\[1ex]
&=
b_{w-3,w-3}\frac{(2b_{w-3,w-2} - (b_{w-3,w-3}+1))}{(b_{w-3,w-2}+b_{w-3,w-3}+1)}\frac{\left(\frac{(b_{w-1,w-1}b_{w-3,w-2} - (b_{w-3,w-3}+1))(b_{w-3,w-2}+1)}{b_{w-3,w-2}+b_{w-3,w-3}+1} \right)-(b_{w-3,w-2}+1)}{\left(\frac{(b_{w-1,w-1}b_{w-3,w-2} - (b_{w-3,w-3}+1))(b_{w-3,w-2}+1)}{b_{w-3,w-2}+b_{w-3,w-3}+1} \right)}
&&
\nonumber \\[1ex]
&=
b_{w-3,w-3}\frac{(2b_{w-3,w-2} - (b_{w-3,w-3}+1))}{(b_{w-3,w-2}+b_{w-3,w-3}+1)}\frac{\left((2b_{w-3,w-2} - (b_{w-3,w-3}+1)) \right)-(b_{w-3,w-2}+b_{w-3,w-3}+1)}{2b_{w-3,w-2} - (b_{w-3,w-3}+1)}
&&
\nonumber \\[1ex]
&=
b_{w-3,w-3}\frac{\left(2b_{w-3,w-2} - (b_{w-3,w-3}+1) \right)-(b_{w-3,w-2}+b_{w-3,w-3}+1)}{b_{w-3,w-2}+b_{w-3,w-3}+1}
&&
\nonumber \\[1ex]
&=
b_{w-3,w-3} - \frac{3b_{w-3,w-3}(b_{w-3,w-3}+1)}{b_{w-3,w-2}+b_{w-3,w-3}+1}
&&
\nonumber \\[1ex]
&=
b_{w-3,w-3} - \frac{3(b_{w-3,w-3}+1)}{b_{w-2,w-2}+1}
&&
\label{Y-4-cut-holds}
\end{alignat}

\noindent where the final equality follows from the $\mathbf{Y}$-diamond rule: $(0+1)(b_{w-3,w-2}+1) = b_{w-3,w-3}b_{w-2,w-2}$.

Analogously, we also have $\frac{b_{2,w-3}b_{2,w-1}b_{2,w+1}}{(b_{2,w-2}+1)(b_{2,w}+1)} = b_{w+3,w+3}-\frac{3(b_{w+3,w+3}+1)}{b_{w+2,w+2}+1}$. 

Finally, observe that $b_{w-2,w+2} = b_{1,w-4}$ (due to glide symmetry) belongs to the following triangle of entries in $B$:

\[
\begin{matrix}
0 && 0 && 0 && 0 && 0 && 0\\
& b_{w-2,w-2} && 2 && 2 && 2 && b_{w+2,w+2} &\\
&& b_{w-2,w-1}&&b_{w-1,w}&& b_{w,w+1} && b_{w+1,w+2}&&\\
&&& b_{w-2,w} && b_{w-1,w+1} && b_{w,w+2} &&&\\
&&&& b_{w-2,w+1} && b_{w-1,w+2}  &&&&\\
&&&&& b_{1,w-4} &&&&&
\end{matrix}
\]

Using the $\mathbf{Y}$-diamond rule to complete the entries of this triangle from the top down, we obtain: \begin{equation}
\label{Y-4-integral-diamond}
b_{1,w-4} = \frac{(b_{w-2,w-2}-2)(b_{w+2,w+2}-2)}{3}-1.
\end{equation}


Therefore, as desired, the first row of $B'$ arises from a $\mathbf{Y}$-$4$-cut on the first row of $B$.

To prove the $\mathbf{Y}$-diamond rule holds everywhere in $B'$ it suffices to show that the rule holds on each diamond that involves the entry $b'_{i,w-3} := \frac{b_{i,w-3}b_{i,w-1}b_{i,w+1}}{(b_{i,w-2}+1)(b_{i,w}+1)}$ for some $i \in \{1,\ldots, w-3 \}$, since these are the only entries not in $B$. For each $i \in \{1,\ldots, w-3\}$ there are two such types of diamonds:

\[
\begin{matrix}
    &&b_{i+1,w-4}&&\\
    &b_{i,w-4}&& b'_{i+1,w-3}\\
    &&b'_{i,w-3}&&
\end{matrix}
\qquad \text{and} \ \qquad
\begin{matrix}
    &&b'_{i+1,w-3}&&\\
    &b'_{i,w-3}&&b_{1,i-1}\\
    &&b_{1,i-2}
\end{matrix}
\]

As in Proposition \ref{Y-2-cut} and \ref{Y-2-cut}, to show the $\mathbf{Y}$-diamond rule holds on the left diamond, one may equivalently show that the entries in $B'$ satisfy equation (\ref{y-diagonals2}) for $f_n := \frac{b_{i,w-3}b_{i,w-1}b_{i,w+1}}{(b_{i,w-2}+1)(b_{i,w}+1)}$. We now verify these entries do indeed satisfy equation (\ref{y-diagonals2}):

\begin{alignat}{2}
&
\frac{\left(\left(\frac{b_{w-3,w-3}b_{w-3,w-1}b_{w-3,w+1}}{(b_{w-3,w-2}+1)(b_{w-3,w-2}+1)}\right)b_{i,w-4}-(b_{i,w-5}+1)\right)
(b_{i,w-4}+1)}
{b_{i,w-4}+b_{i,w-5}+1}
&&
\nonumber\\[1ex]
&=
\frac{\left(\left(b_{w-3,w-3} - \frac{3(b_{w-3,w-3}+1)}{b_{w-2,w-2}+1}\right)b_{i,w-4}-(b_{i,w-5}+1)\right)
(b_{i,w-4}+1)}
{b_{i,w-4}+b_{i,w-5}+1}
&&
\qquad \text{by equality (\ref{Y-4-cut-holds})}
\nonumber\\[1ex]
&=
b_{i,w-3}
-\frac{3(b_{w-3,w-3}+1)b_{i,w-4}(b_{i,w-4}+1)}
{(b_{w-2,w-2}+1)(b_{i,w-4}+b_{i,w-5}+1)}
&&
\qquad \parbox[t]{7cm}{\raggedright
by Proposition \ref{y-diagonals}\\
applied to $f_n:=b_{i,w-3}$}
\nonumber \\[1ex]
&=
b_{i,w-3}
-\frac{3(b_{i,w-3}+ b_{i,w-4}+1)}
{b_{w-2,w-2}+1}
&&
\qquad \parbox[t]{7cm}{\raggedright
by Corollary \ref{reduction}\\
applied to $f_n:=b_{i,w-3}$}
\nonumber\\[1ex]
&=
b_{i,w-3}
-\frac{3b_{i,w-3}(b_{i,w-3}+1)}
{b_{i,w-2}+ b_{i,w-3}+1}
&&
\qquad \parbox[t]{7cm}{\raggedright
by Corollary \ref{reduction}\\
applied to $f_n:=b_{i,w-2}$}
\nonumber\\[1ex]
&=
\frac{b_{i,w-3}(b_{i,w-2} - 2(b_{i,w-3}+1))}
{b_{i,w-2}+ b_{i,w-3}+1}
&&
\nonumber \\[1ex]
&=
\frac{b_{i,w-3}b_{i,w-1}(b_{i,w-2} - 2(b_{i,w-3}+1))}
{(2b_{i,w-2}- (b_{i,w-3}+1))(b_{i,w-2}+1)}
&&
\hspace{-1.8cm}\parbox[t]{7cm}{\vspace{-0.7cm}\raggedright
by $b_{w-1,w-1}=2$ and Proposition \ref{y-diagonals} applied to $f_n:=b_{i,w-1}$.}
\label{Y-4-frieze-rule-holds}
\intertext{On the other hand, note that we have:}
&
\frac{b_{i,w+1}}{b_{i,w}+1} =
\frac{2b_{i,w}-(b_{i,w-1}+1)}{b_{i,w}+b_{i,w-1}+1}
&&
\hspace{-1.8cm}\parbox[t]{7cm}{\vspace{-0.7cm}\raggedright
by $b_{w-1,w-1}=2$ and Proposition \ref{y-diagonals} applied to $f_n:=b_{i,w+1}$.}
\nonumber \\[1ex]
&= \frac{2\left(\frac{(2b_{i,w-1} - (b_{i,w-2}+1))(b_{i,w-1}+1)}{b_{i,w-1}+b_{i,w-2}+1}\right)-(b_{i,w-1}+1)}{\left(\frac{(2b_{i,w-1} - (b_{i,w-2}+1))(b_{i,w-1}+1)}{b_{i,w-1}+b_{i,w-2}+1}\right)+(b_{i,w-1}+1)}
&&
\hspace{-1.8cm}\parbox[t]{6.5cm}{\vspace{-0.5cm}\raggedright
by $b_{w,w}=2$ and Proposition \ref{y-diagonals} applied to $f_n:=b_{i,w}$.}
\nonumber \\[1ex]
&= \frac{b_{i,w-1}-(b_{i,w-2}+1)}{b_{i,w-1}}
&&
\nonumber \\[1ex]
&= \frac{\left(\frac{(2b_{i,w-2} - (b_{i,w-3}+1))(b_{i,w-2}+1)}{b_{i,w-2}+b_{i,w-3}+1}\right)-(b_{i,w-2}+1)}{\left(\frac{(2b_{i,w-2} - (b_{i,w-3}+1))(b_{i,w-2}+1)}{b_{i,w-2}+b_{i,w-3}+1}\right)}
&&
\hspace{-1.8cm}\parbox[t]{7cm}{\vspace{-0.7cm}\raggedright
by $b_{w-1,w-1}=2$ and Proposition \ref{y-diagonals} applied to $f_n:=b_{i,w-1}$.}
\nonumber \\[1ex]
&= \frac{b_{i,w-2} - 2(b_{i,w-3}+1)}
{2b_{i,w-2}- (b_{i,w-3}+1)}
&&
\nonumber 
\end{alignat}

Hence, as desired, equality (\ref{Y-4-frieze-rule-holds}) may equivalently be written as $$\frac{\left(\left(\frac{b_{w-3,w-3}b_{w-3,w-1}b_{w-3,w+1}}{(b_{w-3,w-2}+1)(b_{w-3,w-2}+1)}\right)b_{i,w-4}-(b_{i,w-5}+1)\right)
(b_{i,w-4}+1)}
{b_{i,w-4}+b_{i,w-5}+1} = \frac{b_{i,w-3}b_{i,w-1}b_{i,w+1}}{(b_{i,w-2}+1)(b_{i,w}+1)}.$$
The $\mathbf{Y}$-diamond rule holds for the second type of diamond analogously.

Clearly $b'_{i,w-3}:=\frac{b_{i,w-3}b_{i,w-1}b_{i,w+1}}{(b_{i,w-2}+1)(b_{i,w}+1)}$ is (strictly) positive for all $2 \leq i \leq w-3$. To argue integrality, we first note that by Proposition \ref{diagonal-divisibility} we have $\frac{b_{i,w+1}b_{i,w-1}}{b_{i,w}+1}, \frac{b_{i,w}b_{i,w-2}}{b_{i,w-1}+1} \in \mathbb{Z}$.

The integrality of $b'_{i,w-3}$ is then a consequence of the following equalities:

\begin{alignat*}{2}
b'_{i,w-3} &+ \frac{b_{i,w+1}b_{i,w-1}}{b_{i,w}+1}
=
\frac{b_{i,w+1}b_{i,w-1}(b_{i,w-2}+b_{i,w-3}+1)}{(b_{i,w-2}+1)(b_{i,w}+1)}
&&
\\[1ex]
&=
\frac{3b_{i,w+1}b_{i,w-1}b_{i,w-2}}{(b_{i,w}+1)(b_{i,w-1}+b_{i,w-2}+1)}
&&
\qquad \parbox[t]{7cm}{\raggedright
by Corollary \ref{reduction}\\
applied to $f_n:=b_{i,w-1}$}
\\[1ex]
&=
\frac{3b_{i,w-1}b_{i,w-2}(3b_{i,w}-(b_{i,w-1}+1))}{(b_{i,w}+b_{i,w-1}+1)(b_{i,w-1}+b_{i,w-2}+1)}
&&
\qquad \parbox[t]{7cm}{\raggedright
by Proposition \ref{y-diagonals}\\
applied to $f_n:=b_{i,w+1}$}
\\[1ex]
&=
\frac{3b_{i,w-2}(3b_{i,w}-(b_{i,w-1}+1))}{3(b_{i,w-1} + 1)}
&&
\qquad \parbox[t]{7cm}{\raggedright
by Corollary \ref{reduction}\\
applied to $f_n:=b_{i,w-1}$}
\\[1ex]
&=
\frac{3b_{i,w-2}b_{i,w}}{b_{i,w-1} + 1}- b_{i, w-2}.
\end{alignat*}

\end{proof}

\begin{lemma}
\label{Y-4-lemma}
Let $B$ be a $\mathbf{Y}$-frieze pattern of width $w$, and let $B'$ be the $\mathbf{Y}$-frieze pattern obtained from a $\mathbf{Y}$-4-cut on $B$ with respect to the indexing $$(b_{w-1},b_{w},b_{w+1}) := (b_{w-1,w-1},b_{w,w},b_{w+1,w+1}) = (2,2,2).$$

Furthermore, let $A'$ be a frieze pattern of width $w-4$ such that $p_{w-4}(A') = B'$, and let $A$ be obtained from $A'$ by a corresponding 4-glue. Let $q(A) = (a_1,\ldots, a_{w+3})$ and $q(A') = (a'_1,\ldots, a'_{w-1})$ be the quiddity cycles of $A$ and $A'$ respectively.
Then the following holds: 
\begin{enumerate}[label=\normalfont(\roman*)]
\item if $A$ is obtained from $A'$ by a 4(a)-glue then $p_w(A) = B$ is equivalent to $b_{w-2} = a'_{w-2}+2$,
\item if $A$ is obtained from $A'$ by a 4(b)-glue then $p_w(A) = B$ is equivalent to $b_{w+2} = a'_{w-1}+2$.
\end{enumerate}
\end{lemma}

\begin{proof}
Directly from Remark \ref{cut polygons} and Proposition \ref{double cut}, if $A$ is obtained from $A'$ by a 4(a)-glue then $q(A)=(a'_1,a'_2,...,a'_{w-3},a'_{w-2}+3,1,3,1,3,a'_{w-1}+1).$ 

As in the proof of Lemma \ref{Y-2-lemma} and \ref{Y-3-lemma}, writing $q(B) = (b_1,\ldots, b_{w+3})$ then $p_{w}(A) = B$ may be restated as:

\begin{align*}
b_i =
\begin{cases}
        a'_ia'_{i+1} - 1, &1\leq i\leq w-4\\[0.5em]
        a'_{w-3}(a'_{w-2} +3)-1, &i=w-3\\[0.5em]
        a'_{w-2}+2, &i=w-2\\[0.5em]
        2, &w-1\leq i\leq w+1\\[0.5em]
        3a'_{w-1}+2, &i=w+2\\[0.5em]
        (a'_{w-1}+1)a'_1 -1, &i=w+3.
\end{cases}
\end{align*}

Note that $b_{w-1} = b_{w} = b_{w+1} = 2$ follows immediately from the conditions of a $\mathbf{Y}$-4-cut. On the other hand, by assumption, we have $p_{w-4}(A')=B'$ which is equivalent to $a'_{i}a'_{i+1} = b'_{i} + 1$ holding for all $1 \leq i \leq w-1$ where $q(B') = (b'_1,\ldots, b'_{w-1})$. By the definition of $\mathbf{Y}$-4-cut, this may be restated as:

    \begin{align*}\label{system4}
    a'_{i}a'_{i+1} = \begin{cases}
        \qquad b_i+1, &1\leq i\leq w-4\\[0.5em] 
        \frac{(b_{w-3}+1)(b_{w-2}-2)}{b_{w-2}+1}, &i=w-3\\[0.5em] 
        \frac{(b_{w-2}-2)(b_{w+2}-2)}{3}, &i=w-2\\[0.5em] 
        \frac{(b_{w+2}-2)(b_{w+3}+1)}{b_{w+2}+1}, &i=w-1.
        \end{cases}\tag{***}
    \end{align*}

Immediately from (\ref{system4}) we see $b_{i} = a'_ia'_{i+1} - 1$ for all $1\leq i\leq w-4$. Now, let us assume $b_{w-2} = a'_{w-2}+2$. Considering $i=w-3$ and $i=w-2$ from (\ref{system4}) we see \begin{align*} 
a'_{w-3}a'_{w-2} &= \frac{(b_{w-3}+1)(b_{w-2}-2)}{b_{w-2}+1} = \frac{(b_{w-3}+1)a'_{w-2}}{a'_{w-2}+3} \\
a'_{w-2}a'_{w-1} & = \frac{(b_{w-2}-2)(b_{w+2}-2)}{3} = \frac{a'_{w-2}(b_{w+2}-2)}{3}
\end{align*}

Consequently, $b_{w-3} = a'_{w-3}(a'_{w-2}+3) - 1$ and $b_{w+2} =  3a'_{w-1} + 2$. Similarly, considering $i = w-1$ from (\ref{system4}) and using $b_{w+2} =  3a'_{w-1} + 2$ we see: $$a'_{w-1}a'_{1} = \frac{(b_{w+2}-2)(b_{w+3}+1)}{b_{w+2}+1} = \frac{a'_{w-1}(b_{w+3}+1)}{a'_{w-1}+1}$$

Hence, as desired, we also have $b_{w+3} = (a'_{w-1} +1)a'_{1}-1$.

The proof of the case when $A$ is obtained from $A'$ by a 2(b)-glue follows analogously.    

\end{proof}

\begin{proposition}\label{4-surjection}
    Let $B$ be a $\mathbf{Y}$-frieze pattern of width $w$, and let $B'$ be the $\mathbf{Y}$-frieze pattern obtained from a $\mathbf{Y}$-4-cut on $B$, as in Proposition \ref{Y-4-cut}. If $B'$ is in the range of $p_{w-4}$, then $B$ is in the range of $p_w$.
\end{proposition}

\begin{proof}
    To prove the proposition, we will show there is a frieze pattern $A'$ of width $w-4$ such that: $$ p_{w-4}(A')=B' \qquad \text{and} \qquad p_{w}(A)=B$$
\noindent where $A$ is a frieze pattern obtained by performing a $4$-glue on $A'$.

Directly from Lemma \ref{Y-4-lemma}, in order to show this, it suffices to find a frieze pattern $A'$ such that $p_{w-4}(A') = B'$ and:
    \begin{enumerate}[label=\normalfont(\roman*)]
\item $b_{w-2} = a'_{w-2}+2$, if $A$ is obtained from $A'$ by a 2(a)-glue,
\item $b_{w+2} = a'_{w-1}+2$, if $A$ is obtained from $A'$ by a 2(b)-glue. 
    \end{enumerate}

Note that, as highlighted in the proof of Proposition \ref{2-surjection}, care is required when choosing $A'$ if $B$ has odd width. Indeed, Example \ref{Y-4-bad_choice} demonstrates it is possible that $p_{w-4}(A') = B'$ but $p_{w}(A)\neq B$ for any choice of 4-glue relating $A'$ to $A$. We shall thus divide our proof into two cases by the parity of $w+1$.

    \textbf{Case 1}: $w+1$ is even.

    Based on the conditions required for a 4(a) and 4(b)-glue, we will now construct two rational sequences $(d_1,\ldots, d_{w-1})$ and $(c_1,\ldots, c_{w-1})$ and argue that one of them consists entirely of integers, and is our desired $q(A')$.

    Setting $c_{w-1}:=b_{w+2}-2$, we then inductively generate the remaining $c_i$ for $w-2 \geq i \geq 1$ using (\ref{system4}). Doing so, we obtain: 
    
\[
c_i :=
\begin{cases}
\dfrac{b_{w-2}-2}{b_{w-1}+1},
& \text{if } i=w-2,\\[2ex]
\dfrac{(b_i+1)(b_{i+2}+1)\cdots(b_{w-3}+1)(b_{w-1}+1)}
{(b_{i+1}+1)(b_{i+3}+1)\cdots(b_{w-2}+1)},
& \text{if } 2\le i\le w-3 \text{ and } i \text{ is even},\\[3ex]
\dfrac{(b_i+1)(b_{i+2}+1)\cdots(b_{w-2}+1)}
{(b_{i+1}+1)(b_{i+3}+1)\cdots(b_{w-3}+1)(b_{w-1}+1)},
& \text{if } 1\le i\le w-4 \text{ and } i \text{ is odd},
\end{cases}
\]

\noindent where $b_{w-1}+1 = 3$. Furthermore, note that $(c_1,\ldots, c_{w-1})$ satisfies (\ref{system4}). Indeed, this follows from construction and the equality $\frac{(b_1+1)(b_{3}+1)\cdots(b_{w-2}+1)}
{(b_{2}+1)(b_{4}+1)\cdots(b_{w-3}+1)(b_{w-1}+1)} = \frac{b_{w+3}+1}{b_{w+2}+1}$. In particular, this equality is implied by Proposition \ref{product} applied to $B'$, and the use of (\ref{system4}). We stress this is a valid application of Proposition \ref{product} from the assumption that $p_{w-4}(A') = B'$ for some frieze $A'$.

Similarly, setting $d_{w-2}:=b_{w-2}-2$ and following (\ref{system4}) yields:

\[
d_i :=
\begin{cases}
\dfrac{(b_i+1)(b_{i+2}+1)\cdots(b_{w-2}+1)}
{(b_{i+1}+1)(b_{i+3}+1)\cdots(b_{w-3}+1)},
& \text{if } 1\le i\le w-4 \text{ and } i \text{ is odd},\\[3ex]
\dfrac{(b_i+1)(b_{i+2}+1)\cdots(b_{w-3}+1)}
{(b_{i+1}+1)(b_{i+3}+1)\cdots(b_{w-2}+1)},
& \text{if } 2\le i\le w-3 \text{ and } i \text{ is even},\\[3ex]
\dfrac{(b_{w+2}-2)(b_{w+3}+1)(b_{2}+1)(b_{4}+1)\cdots(b_{w-3}+1)}
{(b_{w+2}+1)(b_{1}+1)(b_{3}+1)\cdots(b_{w-2}+1)} = \dfrac{b_{w+2}-2}{3},
& \text{if } i=w-1,\\[2ex]
\end{cases}
\]

\noindent where the equality $\frac{(b_{w+2}-2)(b_{w+3}+1)(b_{2}+1)(b_{4}+1)\cdots(b_{w-3}+1)}
{(b_{w+2}+1)(b_{1}+1)(b_{3}+1)\cdots(b_{w-2}+1)} = \frac{b_{w+2}-2}{3}$ follows from Proposition \ref{product} applied to $B'$, and the use of (\ref{system4}). Consequently, $(d_1,\ldots, d_{w-1})$ satisfies (\ref{system4}).

Note that we have the following, where the integrality follows by Proposition \ref{divisibility}:
\begin{equation}
\label{Y-4-oscillating}
\left\{
\begin{aligned}
&3d_i=c_i\in\mathbb{Z}, &&\text{if $i$ is even},\\
&3c_i=d_i\in\mathbb{Z}, &&\text{if $i$ is odd}.
\end{aligned}
\right.
\end{equation} 



Now, consider the following sequence: \begin{equation}
\label{Y-4-candidate-sequence}
d_{w-2},\; c_{w-2},\; c_{w-3},\; d_{w-3},\; d_{w-4},\;\ldots,\; c_2,\; d_2,\; d_1,\; c_1,\; c_{w-1},\; d_{w-1}.
\end{equation}

By (\ref{Y-4-oscillating}), if a given term $d_{2k+1}$ or $c_{2k}$ is divisible by $3$ then $c_{2k+1}$ or $d_{2k}$ is an integer, respectively. Hence, moving down the sequence (\ref{Y-4-candidate-sequence}), all entries are guaranteed to be integers until we meet a term $d_{2k+1}$ or $c_{2k}$ that is not divisible by $3$. On the other hand, for any $0 \leq k \leq \frac{w-3}{2}$, if $d_{2k+1} \not\equiv 0 \mod 3$, then every subsequent term $d_{2k+1},d_{2k},\ldots, d_1 ,d_{w-1}$ is an integer. 

Indeed, for each such subsequent term $d_i$, if $i$ is odd then $d_i \in \mathbb{Z}$ by (\ref{Y-4-oscillating}), so it remains to consider when $i$ is even. In that case, by Proposition \ref{divisibility} the product $d_i\cdot d_{2k+1} \in \mathbb{Z}$ for $2k+1 \geq i \geq 2$, and an integer multiple of $\frac{(b_{w+2}-2)}{(b_{w+2}+1)}$ if $i = w-1$. Furthermore, since $3d_i \in \mathbb{Z}$ by (\ref{Y-4-oscillating}), then $d_i \in \mathbb{Z}$ by the assumption that $\gcd(d_{2k+1},3) = 1$.

%
%

%
%
%

Likewise, for any $1 \leq k \leq \frac{w-3}{2}$, if $c_{2k} \not\equiv 0 \mod 3$, then every subsequent term $c_{2k},c_{2k-1},\ldots,c_1, c_{w-1}$ is an integer. 

Therefore, as desired, either $(c_1,\ldots, c_{w-1})$ or $(d_1,\ldots, d_{w-1})$ is an integer sequence, and we choose $q(A'): = (a'_1,\ldots, a'_n)$ to be one such sequence.

    Finally, by assumption, since there exists $(a''_i)_{i\in\mathbb{Z}}$ that also satisfies equations (\ref{system4}), then our integer choice of $(a'_i)_{i\in\mathbb{Z}}$ is a frieze pattern by Proposition \ref{ratio-classification}. This completes the proof of case 1.

    \textbf{Case 2}: $w+1$ is odd.

    Since $w+1$ is odd and by assumption the system of equations (\ref{system4}) has a solution, there is a unique positive solution $(a'_i)_{i\in\mathbb{Z}}$.

    We can compute ${a'_{w-2}}^2$ by considering ${a'_{w-2}}^2=\frac{(a'_1a'_2)\cdots (a'_{w-3}a'_{w-2})(a'_{w-2}a'_{w-1})}{(a'_{w-1}a'_1)(a'_2a'_3)\cdots (a'_{w-4}a'_{w-3})}$ and using the equations (\ref{system4}) to get: 

    \[{a'_{w-2}}^2=\frac{(b_{w+2}+1)(b_1+1)(b_3+1)\cdots(b_{w-3}+1)(b_{w-2}-2)^2}{3(b_{w+3}+1)(b_2+1)(b_4+1)\cdots(b_{w-4}+1)(b_{w-2}+1)}\]

    Consider $q=\frac{(b_{w+2}+1)(b_1+1)(b_3+1)\cdots(b_{w-3}+1)}{(b_{w+3}+1)(b_2+1)(b_4+1)\cdots(b_{w-4}+1)(b_{w-2}+1)}$. Since $b_{w-1}+1=3$ then $$3q=\frac{(b_{w+2}+1)(b_1+1)(b_3+1)\cdots(b_{w-3}+1)(b_{w-1}+1)}{(b_{w+3}+1)(b_2+1)(b_4+1)\cdots(b_{w-4}+1)(b_{w-2}+1)}$$ is a positive integer by Proposition \ref{divisibility}. On the other hand, since $b_{w+1}+1=3$ then $$\frac{3}{q}=\frac{(b_{w+1}+1)(b_{w+3}+1)(b_2+1)(b_4+1)\cdots(b_{w-4}+1)(b_{w-2}+1)}{(b_{w+2}+1)(b_1+1)(b_3+1)\cdots(b_{w-3}+1)}$$ is also a positive integer by Proposition \ref{divisibility}. Hence $q \in \{\frac{1}{3}, 1, 3\}$.

    If $q=\frac{1}{3}$, then $a'_{w-2}=\frac{b_{w-2}-2}{3}$ (so $a'_{w-1}=b_{w+2}-2$), and if $q=3$, then $a'_{w-2}=b_{w-2}-2$. If $q=1$, then $a'_{w-2}=\frac{b_{w-2}-2}{\sqrt{3}}$ which is absurd since $a'_{w-2}$ is a frieze pattern entry and hence must be an integer.

    In summary, for our frieze pattern $A'$ satisfying $p_{w-4}(A') = B'$ we have either $a'_{w-2}=b_{w-2}-2$ or $a'_{w-1}=b_{w+2}-2$. Therefore, Lemma \ref{Y-4-lemma} guarantees $p_w(A) = B$, and the proof of case 2 is complete.
    
\end{proof}

\begin{example}\label{Y-4-bad_choice}
Consider the $\mathbf{Y}$-frieze $B$ of width $w = 9$ with $q(B)=(3,3,3,3,1,8,8,2,2,2,5,3)$. Then $q(B')=(3,3,3,3,1,5,5,1)$ and there are two friezes $A'$ such that $q(A') = B'$; the ``good'' choice of $A'$ has $q(A')=(2,2,2,2,1,6,1)$ and the ``bad'' choice has $q(A')=(1,4,1,4,1,2,3,2)$. Moreover, for the ``bad'' choice of $A'$, one can easily verify there is no frieze $A$ arising from a $4$-glue on $A'$ such that $p_{w}(A) = B$.
\end{example}

\begin{theorem}\label{main}
    Every $\mathbf{Y}$-frieze pattern of width $w$ is the image, under $p_w$, of a frieze pattern of width $w$.
\end{theorem}

\begin{proof}

    We prove the result by induction on $w$.
    
    The cases where $w=-1,0$, or $1$ are trivial. The statement also holds for $w=2$ since the $\mathbf{Y}$-diamond rule and diamond rule coincide for the nontrivial entries.

    Now, suppose the statement holds for all $w'<w$. Let $B$ be a $\mathbf{Y}$-frieze pattern of width $w$. By Proposition \ref{can cut}, there exists a subsequence that makes a $\mathbf{Y}$-$n$-cut possible on $B$ for some $n \in \{2,3,4\}$. Let $B'$ be the resulting $\mathbf{Y}$-frieze pattern resulting from such an $n$-cut. By our inductive hypothesis, $B'$ is in the range of $p_{w-n}$. Therefore, applying Propositions \ref{Y-2-cut}, \ref{Y-3-cut}, and \ref{Y-4-cut}, we see that $B$ is in the range of $p_w$, as desired.
\end{proof}

Combining Theorem \ref{Y-frieze from frieze numnber} and Theorem \ref{main} leads to the enumeration of all $\mathbf{Y}$-frieze patterns.

\begin{corollary}
\label{main-enumeration}
Let $\lvert\operatorname{YFrieze}(w)\rvert$ denote the number of $\mathbf{Y}$-frieze patterns of width $w$. Then we have:

\[
\lvert\operatorname{YFrieze}(w)\rvert = \begin{cases}
C_{w+1}^{(2)},
&\text{if $w$ is even};\\[1.5em]
C_{w+1}^{(2)}-C_{\frac{w+1}{2}}^{(3)},
&\text{if $w\equiv1\text{ }(\operatorname{mod}4)$};\\[1.5em]
C_{w+1}^{(2)}-C_{\frac{w+1}{2}}^{(3)}-C_{\frac{w+1}{4}}^{(5)},
&\text{if $w\equiv3\text{ }(\operatorname{mod}4)$}.
\end{cases}
\]

\noindent where $C_{p}^{(q)}:=\frac{1}{(q-1)p+1}{qp\choose p}$ are the Fuss-Catalan numbers.
\end{corollary}

\printbibliography

\end{document}